\def\draft{0}
\def\done{1}

\if \draft1
\documentclass[draft]{amsart}
\else
\documentclass{amsart}
\fi
\usepackage{amsmath, amsthm, amsfonts, amsrefs, amssymb}
\usepackage{colonequals}
\usepackage{aliascnt}
\usepackage[colorlinks=true,linkcolor=red!40!black,citecolor=green!40!black,urlcolor=cyan!40!black]{hyperref}
\usepackage{tikz}
\usetikzlibrary{calc}
\usepackage[matrix, arrow]{xy}
\entrymodifiers={+!!<0pt,\fontdimen22\textfont2>}

\usepackage{verbatim}

\if \draft1
\usepackage[color]{showkeys}
\fi

\if \done0
\usepackage{todonotes}

\usepackage{everypage}
\usepackage[absolute]{textpos}
\AddEverypageHook{
\begin{textblock}{1}(0,1.5)
\centering
\color{red}
Draft: \today
\end{textblock}
}
\fi

\title{Khovanov-Rozansky Homology and Conway Mutation}
\author{Thomas C. Jaeger}
\address{Department of Mathematics;
Michigan State University;
East Lansing, MI 48824}
\email{tjaeger@math.msu.edu}

\newtheorem{theorem}{Theorem}[section]

\newcommand\newthm[2]{
\newaliascnt{#1}{theorem}
\newtheorem{#1}[#1]{#2}
\aliascntresetthe{#1}
\expandafter\newcommand\csname #1autorefname\endcsname{#2}
}

\newthm{lemma}{Lemma}
\newthm{proposition}{Proposition}
\newthm{corollary}{Corollary}
\theoremstyle{definition}
\newthm{definition}{Definition}
\newthm{remark}{Remark}

\tikzstyle{under} = [above=-1.35em]
\tikzstyle{upper} = [preaction={draw=white, -, line width=4pt}]
\tikzstyle{ul} = [anchor=south east, text height=1.5ex, text depth=0.25ex, yshift=-0.25ex, inner sep=0.1em]
\tikzstyle{ur} = [anchor=south west, text height=1.5ex, text depth=0.25ex, yshift=-0.25ex, inner sep=0.1em]
\tikzstyle{ll} = [anchor=north east, text height=1.5ex, text depth=0.25ex, inner sep=0.1em]
\tikzstyle{lr} = [anchor=north west, text height=1.5ex, text depth=0.25ex, inner sep=0.1em]

\tikzstyle{cl} = [anchor= east, text height=1.5ex, text depth=0.25ex]
\tikzstyle{cr} = [anchor= west, text height=1.5ex, text depth=0.25ex]

\tikzstyle{t} = [anchor=south, text height=1.5ex, text depth=0.25ex]
\tikzstyle{b} = [anchor=north, text height=1.5ex, text depth=0.25ex]

\DeclareMathOperator{\Hom}{Hom}
\DeclareMathOperator{\Cone}{Cone}
\DeclareMathOperator{\im}{im}
\newcommand{\id}{\mathrm{id}}

\newcommand\mmatrix[4]{\begin{pmatrix} #1 & #2 \\ #3 & #4 \end{pmatrix}}
\newcommand\smatrix[1]{\scalebox{0.8}{$\begin{pmatrix}#1\end{pmatrix}$}}

\newcommand\koszul[1]{\begin{Bmatrix} #1 \end{Bmatrix}}
\newcommand\mf[4]{\xymatrix{#1 \ar@<0.4ex>[r]^-{#3} & #2 \ar@<0.4ex>[l]^-{#4}}}
\newcommand\mysl[1]{\mathfrak{sl}(#1)}
\newcommand\sln{\mysl n}
\newcommand\mfshift{\langle1\rangle}
\newcommand\hmf{hmf}

\newcommand\KR{Khovanov-Rozansky }

\newcommand\Z{\mathbb Z}
\newcommand\Q{\mathbb Q}
\newcommand\K{\mathrm K^{b}}

\newcommand\FF{\mathcal F}
\newcommand\GG{\mathcal G}
\newcommand\CC{\mathcal C}
\newcommand\inv{^{-1}}

\newcommand\ce{\colonequals}

\allowdisplaybreaks

\begin{document}

\begin{abstract}
We show that the reduced $\sln$ homology defined by Khovanov and
Rozansky is invariant under component-preserving positive mutation when $n$ is
odd.
\end{abstract}

\maketitle

\section{Introduction}

Understanding the behavior of Khovanov homology under Conway mutation has been
an active area of study.  Wehrli~\cite{wehrli1} demonstrated that unlike the
Jones polynomial, Khovanov homology detects mutation of links.
Bar-Natan~\cite{bar-natan-mutation} showed that for a pair of mutant knots (or,
more generally, two links that are related by component-preserving mutation)
there are two spectral sequences with identical $E_2$ pages converging to the
Khovanov homologies of the knots.  Champanerkar and
Kofman~\cite{champanerkar-kofman} relate Khovanov homology to a
(mutation-invariant) matroid obtained from the Tait graph of a knot diagram.
The question remains open, but with coefficients in $\Z_2$ it was solved
independently by Bloom~\cite{bloom} and Wehrli~\cite{wehrli2}.  In fact, Bloom
proves the more general result that odd Khovanov homology (see Ozsv\'ath,
Rasmussen and Szab\'o~\cite{odd-khovanov}) is invariant under arbitrary
mutation of links.  A similar statement cannot hold for the original Khovanov
homology, as we know from Wehrli's example in \cite{wehrli1}.  Recently,
Kronheimer, Mrowka and Ruberman~\cite{instanton} showed that the total rank of
instanton knot homology is invariant under genus-$2$ mutation, which implies
invariance under Conway mutation.

In this paper, we investigate the effect of mutation on $\sln$ homology, a
generalization of Khovanov homology.  $\sln$ homology, defined by Khovanov and
Rozansky in \cite{khr1}, is a categorification of the $\sln$ polynomial, a
certain specialization of the HOMFLY-PT polynomial that can be obtained from
the fundamental $n$-dimensional representation of $U_q(\sln)$.
As noted implicitly by Gornik~\cite{gornik} and later used by Rasmussen in
\cite{rasmussen} (see also Krasner~\cite{krasner-equivariant} and Wu~\cite{wu-s}), the
definitions make sense in a more general context:
To any polynomial $p \in \Q[x]$, one can assign a homology theory that
conjecturally only depends on the multiplicities of the (complex) roots of
$p'(x)$.  $\sln$ homology is recovered by setting $p(x) = \frac1{n+1}x^{n+1}$.
For odd $n$, we establish invariance under positive mutation, that is mutation
that respects the orientations of both $2$-tangles involved in it.

\begin{theorem}\label{main}
If $L$ and $L'$ are two links related by component-preserving positive mutation
and $n$ is odd, then their reduced $\sln$ homologies are isomorphic (reduced
with respect to the component of the mutation, to be defined in
\autoref{section-definitions}).  More generally, let $p(x) = \sum_k a_{2k}
x^{2k}$ be a polynomial with only even powers of $x$, then the reduced \KR
homologies of $L$ and $L'$ associated to this polynomial are isomorphic.
\end{theorem}

Using Rasmussen's spectral sequence from HOMFLY-PT homology to $\sln$ homology,
and the fact that HOMFLY-PT homology of knots is finite dimensional, we get
the following
\begin{corollary}\label{homfly-pt}
If $K$ and $K'$ are two knots related by positive mutation, then their
HOMFLY-PT homologies are isomorphic.
\end{corollary}

We prove the theorem by first showing that the \KR complex of the inner
$2$-tangle can be built out of the complexes assigned to two basic diagrams: a
pair of arcs and a singular crossing.  For the \KR complex, we follow
Rasmussen's definitions from \cite{rasmussen}, since Khovanov and Rozansky's
original definitions are not general enough to serve our purpose.  We then
derive a criterion for a certain mapping cone of this complex to be invariant
under reflection, which turns out to be the case for odd $n$ in the case of
positive mutation.  Closing up the tangle, we see that the mapping cone
computes reduced \KR homology.

The methods used in the proof are fairly general, and we envision that
they can be used to show mutation invariance in other contexts.
\begin{itemize}
\item They apply to (unreduced) $\Z_2$ Khovanov homology and more generally to
the equivariant theory over the ring $\Z_2[h,t]$ defined by Khovanov in
\cite{khovanov-frobenius}. In both cases we get mutation invariance under
arbitrary component-preserving mutation, regardless of whether mutation is
positive or negative.
\item They also apply to Khovanov's (integral) $\mysl 3$ homology and Mackaay
and Vaz's corresponding equivariant theory over the ring
$\Z[a,b,c]$~\cite{mackaay-vaz}.  We plan to show in a later paper that the
corresponding theory over $\Z_2[a,b,c]$ is invariant under component-preserving
mutation and that an appropriately defined reduced theory over $\Z[b]$ is
invariant under positive component-preserving mutation.
\item We expect that analogs of the previous statements hold for larger $n$:
For odd $n$, a reduced version of Krasner's equivariant $\sln$ homology~\cite{krasner-equivariant}
should be invariant under positive component-preserving mutation when setting
the variables corresponding to coefficients of odd powers of $p$ to $0$.
For arbitrary $n$, we expect invariance under component-preserving mutation
when working with coefficients in $\Z_2$.  Note, however, that the standard
definition of $\sln$ homology only works with coefficients in $\Q$.
Krasner~\cite{krasner-integral} proposed a definition of an integral version of $\sln$ homology.
\item Equivariant versions of $\sln$ homology give rise to spectral sequences
that can be used to define analogs of Rasmussen's (integer-valued)
$s$-invariant~\cite{rasmussen-s} for $\sln$ homology.  Generalized
$s$-invariants have their source in Gornik's work~\cite{gornik} and have been
studied by Lobb~\cite{lobb} and Wu~\cite{wu-s}.  One can define such an
invariant for any polynomial over $\mathbb C$ of degree $n$ with only single
roots, but it is not known whether the invariant depends on this choice of
polynomial.  We expect that for odd $n$, they are invariant under positive
component-preserving mutation --- at least for a particular choice of
polynomial.  Furthermore, we hope to show invariance of the original
$s$-invariant by considering the equivariant $\mysl 2$ theory over $\Z_2$.
\item We also expect that Khovanov and Rozansky's HOMFLY-PT
homology~\cite{khr2} is invariant under arbitrary mutation.
A different set of technical difficulties arises when studying this question;
we hope to return to the question in a future paper.
\item More generally, we expect Rasmussen's spectral sequence from HOMFLY-PT
homology to $\sln$ homology to be invariant for odd $n$, although it is less
clear how to apply our technique since the data the spectral sequence is
constructed from lacks a satisfactory equivalent for tangles.
\item It will be interesting to see if a generalization of our method can be
applied to show invariance of $\Z_2$ Khovanov homology under genus-$2$ mutation~\cite{genus-2}.
\end{itemize}

Finally, we note that our result is consistent with calculations for the
Kinoshita-Terasaka knot and the Conway knot carried out by
Mackaay and Vaz~\cite{mackaay-vaz-calculation}.

\bigskip

The organization of this paper is as follows.
In \autoref{section-definitions}, we review relevant definitions and explain
our conventions.
In \autoref{section-topology}, we reduce the problem to the case of mutation of
a $2$-tangle in what we call braid form.
In \autoref{section-reflection}, we investigate how the \KR complex behaves
under positive mutation.
In \autoref{section-2-tangles}, we show how to represent the \KR complex of a $2$-tangle
in braid-form as a complex over a particularly simple category. 
In \autoref{section-mutation}, we derive a general criterion for when a chain
complex over an additive category is isomorphic to its image under a certain
involution functor and show how this criterion applies to the problem at hand.
In \autoref{section-proof}, we combine the results from the previous sections to
prove \autoref{main}.

\medskip \textbf {Acknowledgments.}
The author would like to thank Matt Hedden for many helpful discussions and
Mikhail Khovanov for his valuable comments on an earlier version of this paper.
He would like to thank his adviser Effie Kalfagianni for her expertise,
support and encouragement. This
research was in part supported by NSF Grants
\href{http://www.nsf.gov/awardsearch/showAward.do?AwardNumber=0456155}{DMS-0456155} and
\href{http://www.nsf.gov/awardsearch/showAward.do?AwardNumber=0805942}{DMS-0805942}.

\section{Definitions}\label{section-definitions}

\begin{figure}[t]
\centering
\begin{tikzpicture}[thick]
\draw[semithick,->] (-1.2,0)--(1.2,0) node[below] {$x$};
\draw[semithick,->] (0,-1.2)--(0,1.2) node[right] {$y$};
\draw (0,0) circle (1cm);
\draw (  45:0.9)--(  45:1.1) node[ur] {$d$};
\draw ( 135:0.9)--( 135:1.1) node[ul]  {$c$};
\draw ( -45:0.9)--( -45:1.1) node[lr] {$b$};
\draw (-135:0.9)--(-135:1.1) node[ll]  {$a$};
\end{tikzpicture}
\caption{Placement and labeling of the endpoints of the inner tangle}
\label{figure-endpoints}
\end{figure}
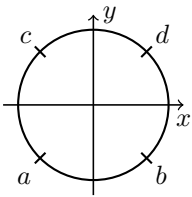

Conway mutation is the process of decomposing a link $L$ as the union of two
2-tangles $L = T \cup T'$ and then regluing in a certain way.  Diagrammatically,
we may assume that one of the tangles (the `inner' tangle $T$) lies inside a
unit circle with endpoints equally spaced as in \autoref{figure-endpoints}.
Mutation consists of one of the following transformations $R$ of the inside
tangle, followed by regluing: reflection along the x-axis ($R_x$), reflection
along the y-axis ($R_y$) or rotation about the origin by 180 degrees ($R_z$).
In other words, the mutant is given by $L' = R(T) \cup T'$.  When taking orientations
into account, we can distinguish two types of mutation (see for example 
Kirk and Livingston~\cite{kirk-livingston}).

\begin{definition}
Mutation of an oriented link is called \emph{positive} if orientations
match when regluing, i.e.\ if $L' = R(T) \cup T'$ as an oriented link
and it is called \emph{negative} if the orientation of the
inner tangle needs to be reversed before regluing, i.e.\ if
$L' = -R(T) \cup T'$ as an oriented link, where $-R(T)$
denotes $R(T)$ with orientations reversed.
\end{definition}
As an example, consider the two knots in \autoref{figure-mutantpair}.
$11^n_{34}$ is a positive mutant of $11^n_{42}$, since rotation
about the $y$-axis preserves the orientations of the ends of $T$.
It is also a negative mutant, as can be seen by considering 
rotation about the $z$-axis.

There are $16$ mutant pairs with $11$ or fewer crossings, see \cite{links-gould}.
It can be checked that all of them can be realized by negative mutation.
Among the $16$ pairs, we found $5$ that can be realized on the tangle $T$ depicted in
\autoref{figure-mutantpair}(a):
$(11^a_{ 57}, 11^a_{231})$,
$(11^a_{251}, 11^a_{253})$,
$(11^n_{ 34}, 11^n_{ 42})$,
$(11^n_{ 76}, 11^n_{ 78})$ and
$(11^n_{151}, 11^n_{152})$.
$R_y(T)$ is isotopic to $T$ but with orientations reversed, therefore these
$5$ mutant pairs can be realized by both positive and negative
mutation.  In particular, our proof applies to
the famous Kinoshita-Terasaka - Conway pair, illustrated in \autoref{figure-mutantpair}(b) and (c).

\newcommand\ktc[2]{
\begin{tikzpicture}[thick,baseline=-0.1cm]
\if 1#2
\path[use as bounding box] (-1.2,-1.8) rectangle (3.2,1.8);
\fi

\if 1#2
\coordinate (x) at (1.8,1.5);
\coordinate (y) at (1.8,-1.5);
\coordinate (a) at (1.5,0.8);
\coordinate (b) at (3,-0.5);
\coordinate (c) at (1.3,0);
\coordinate (d) at (1.8,1.2);
\coordinate (e) at (2.3,0);
\coordinate (f) at (2,-0.6);

\draw (f) to[out=-150,in=-45] (-45:1.1);
\draw[upper,looseness=1.5] (b) to[out=-90,in=-90] (c);
\draw[upper] (e) to[out=-90,in=0] (y);
\draw[upper,looseness=1.8] (x) to[out=0,in=30] (f);
\draw[upper] (a) to[out=-20,in=90] (b);
\draw[upper] (d) to[out=0,in=90] (e);

\draw (c) to[out=90,in=180] (d);
\draw[upper] (45:1.1) to[out=45,in=160] (a);

\draw (135:1.1) to[out=135,in=180] (x);
\draw (-135:1.1) to[out=-135,in=180] (y);
\fi

\draw (0,0) circle (1.1cm);

\if 0#2
\node[ur] at (  45:1.1) {$d$};
\node[ul]  at ( 135:1.1) {$c$};
\node[lr] at ( -45:1.1) {$b$};
\node[ll]  at (-135:1.1) {$a$};
\fi

\if 0#1
\begin{scope}[xscale=-1]
\else
\begin{scope}[yscale=-1]
\fi
\coordinate (A) at (-0.43,  0.45);
\coordinate (B) at (    0,  0.5);
\coordinate (C) at ( 0.43,  0.45);
\coordinate (D) at (-0.31, -0.39);
\coordinate (E) at ( 0.31, -0.39);
\draw \if 1#2 \if 1#1 [<-] \fi \fi (135: 1.1) to[out=-45,in=-135,looseness=1.2] (0, 0.5);
\draw (0, 0.5) to[out= 45,in=   0,looseness=2]   (0,-0.5);
\fill[fill=white] (B) circle (2pt);
\fill[fill=white] (E) circle (2pt);
\draw \if 1#2 \if 0#1 [->] \else [<-] \fi \fi (-135:1.1) to[out= 45,in= 135,looseness=1.5] (-45:1.1);
\fill[fill=white] (D) circle (2pt);
\draw (0,-0.5) to[out=180,in= 135,looseness=2]   (0, 0.5);
\draw \if 1#2 \if 0#1 [->] \fi \fi (0, 0.5) to[out=-45,in=-135,looseness=1.2] (45:1.1);

\begin{scope}
\fill[fill=white] (A) circle (2pt);
\clip (A) circle (2pt);
\draw (135: 1.1) to[out=-45,in=-135,looseness=1.2] (0, 0.5);
\end{scope}

\begin{scope}
\fill[fill=white] (C) circle (2pt);
\clip (C) circle (2pt);
\draw (0, 0.5) to[out= 45,in=   0,looseness=2]   (0,-0.5);
\end{scope}
\end{scope}
\end{tikzpicture}
}

\begin{figure}
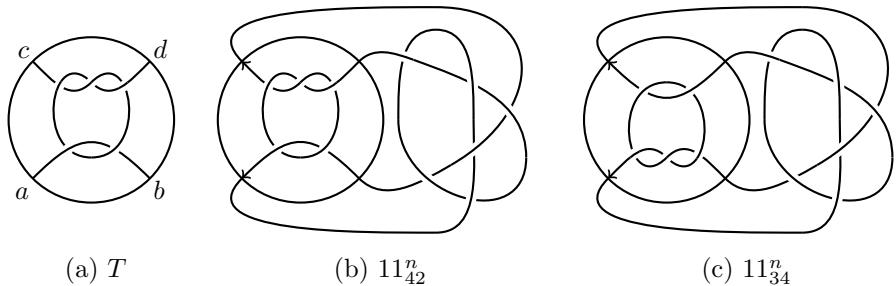

\centering
\begin{tabular}{ccc}
\ktc00 & \ktc01 & \ktc11 \\
(a) $T$ & (b) $11^n_{42}$ & (c) $11^n_{34}$
\end{tabular}
\caption{A 2-tangle with orientation-reversing symmetry and the Kinoshita-Terasaka - Conway mutant pair}
\label{figure-mutantpair}
\end{figure}

\begin{definition}
Mutation of a link is called \emph{component-preserving} if $a$ and
$R(a)$ lie on the same component of the original link (or equivalently,
on the same component of the mutant).
\end{definition}

Note that knot mutation is always component-preserving.  If positive mutation
is component-preserving, then $a$ and $R(a)$ are either both incoming or both
outgoing edges, hence all $4$ endpoints lie on the same component of the link.
We referred to this component earlier as the \emph{component of the mutation}.

For \KR homology, our definitions closely follow \cite{rasmussen}, but note
that we work with $\Z_2$-graded matrix factorizations instead of $\Z$-graded
ones in order to get a stronger version of invariance under
Reidemeister moves.

A matrix factorization over a commutative ring $R$ with potential $w \in
R$ is a free $\Z_2$-graded module $C^*$ equipped with a differential $d = (d_0, d_1)$ such that
$d^2 = w \cdot I_C$.  Following \cite{rasmussen}, we use the notation

\[ \mf{C^1}{C^0}{d^1}{d^0} \]

Morphisms are simply degree-0 maps between matrix factorizations which commute
with the differential.  We denote the category of matrix factorizations over R
with potential $w$ by $MF_w(R)$.  We say that two morphisms of matrix
factorizations $\phi, \psi: C \to C'$ are homotopic if $\phi - \psi = d_{C'}h +
hd_C$ for some degree-1 homotopy $h: C \to C'$.  The category of matrix
factorizations over $R$ with potential $w$ and morphisms considered up to
homotopy will be denoted by $HMF_w(R)$. 
For a graded ring $R$, whose grading we will call \emph{$q$-grading}, we also
introduce a notion of graded matrix factorizations with homogeneous potential $w$
by requiring that both $d^0$ and $d^1$ be homogeneous of $q$-degree $\frac12\deg w$.
Morphisms between graded matrix factorization are required to have $q$-degree $0$, whereas
homotopies must have $q$-degree $-\frac12\deg w$.
The corresponding homotopy category of graded matrix factorizations will be
denoted by $\hmf_w(R)$.  For the three different gradings in $\hmf_w(R)$ we
introduce three types of grading shifts:  A shift in the $\Z_2$ grading coming
from matrix factorizations will be denoted by $\langle\cdot\rangle$, a shift in homological
grading by $[\cdot]$ and a shift in $q$-grading by $\{\cdot\}$.  We follow the convention
that $R[n]$ has a single generator in homological height $n$, and similarly for $\{\cdot\}$.
Note that if $\phi: A \to B$ has $q$-degree $d$, then the $q$-degree of $\phi: A\{k_A\} \to B\{k_B\}$
is $d + k_B - k_A$.

An important class of matrix factorizations is the class of Koszul factorizations, which
we will briefly describe here.  For a more detailed treatment, we refer the reader to
Section 2.2 of \cite{khr3} (but note that we switched the order of the arguments
of $K$ in order to be consistent with \cite{khr2} and \cite{rasmussen}).
If $u,v \in R$, then $K(u;v)$ is the factorization
\[ \mf{R\{\frac{\deg v - \deg u}2\}}{R}{v}{u} \]
We will sometimes write $K_R(u;v)$ to clarify which ring we are working over.
For $\mathbf u = (u_1, ..., u_n)^T$, $\mathbf v = (v_1, ..., v_n)^T$ we define
$K(\mathbf u, \mathbf v) = \bigotimes_{k=1}^n K(u_k; v_k)$.  This
is a matrix factorization with potential $\sum_{k=1}^n a_kb_k$.  We will also use the notation
\[K(\mathbf u, \mathbf v) = \koszul{u_1 & v_1 \\ \vdots & \vdots \\ u_n & v_n}\]
If we are not interested in $\mathbf u$, we may apply arbitrary row transformations
to $\mathbf v$: for an invertible matrix $X$,
$K(\mathbf u, \mathbf v) \cong K\left((X\inv)^t \mathbf u, X \mathbf v\right)$.
We describe order-two Koszul matrix factorizations explicitly, thereby fixing
a sign convention for the tensor product of matrix factorizations:
\[
\koszul{u_1 & v_1 \\ u_2 & v_2} = 
\xymatrix{
R\{k_1\} \oplus R\{k_2\} \ar@<0.4ex>[rr]^-{\mmatrix{u_2}{u_1}{v_1}{-v_2}} &&
R\{k_1+k_2\} \oplus R\ar@<0.4ex>[ll]^-{\mmatrix{v_2}{u_1}{v_1}{-u_2}}
}
\]
Here $k_1 = \frac{\deg v_1 - \deg u_1}2 = \deg v_1 - \frac{\deg w}2$
and  $k_2 = \frac{\deg v_2 - \deg u_2}2 = \deg v_2 - \frac{\deg w}2$.

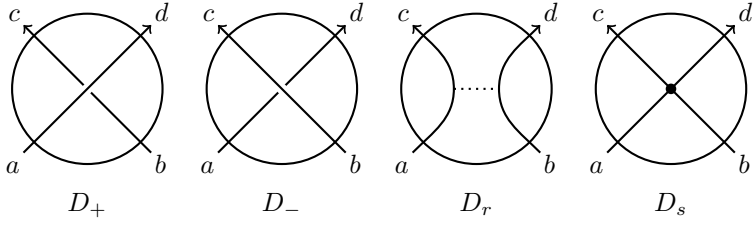
\begin{figure}[t]
\centering
\begin{tabular}{cccc}
\begin{tikzpicture}[thick]
\draw (0,0) circle (1cm);
\draw[->] (135:-1.2) -- (135:1.2);
\fill[fill=white] (0,0) circle (2pt);
\draw[->] ( 45:-1.2) -- ( 45:1.2);
\node[ur] at (  45:1.2) {$d$};
\node[ul] at ( 135:1.2) {$c$};
\node[lr] at ( -45:1.2) {$b$};
\node[ll] at (-135:1.2) {$a$};
\end{tikzpicture}
&
\begin{tikzpicture}[thick]
\draw (0,0) circle (1cm);
\draw[->] ( 45:-1.2) -- ( 45:1.2);
\fill[fill=white] (0,0) circle (2pt);
\draw[->] (135:-1.2) -- (135:1.2);
\node[ur] at (  45:1.2) {$d$};
\node[ul] at ( 135:1.2) {$c$};
\node[lr] at ( -45:1.2) {$b$};
\node[ll] at (-135:1.2) {$a$};
\end{tikzpicture}
&
\begin{tikzpicture}[thick]
\draw (0,0) circle (1cm);
\draw[->] ( 45:-1.2) to[out= 45,in=-90] (-0.3,0) to[out=90,in= -45] (135:1.2);
\draw[->] (135:-1.2) to[out=135,in=-90] ( 0.3,0) to[out=90,in=-135] ( 45:1.2);
\draw[dotted] (-0.3,0)--(0.3,0);
\node[ur] at (  45:1.2) {$d$};
\node[ul] at ( 135:1.2) {$c$};
\node[lr] at ( -45:1.2) {$b$};
\node[ll] at (-135:1.2) {$a$};
\end{tikzpicture}
&
\begin{tikzpicture}[thick]
\draw (0,0) circle (1cm);
\draw[->] ( 45:-1.2)--( 45:1.2);
\draw[->] (135:-1.2)--(135:1.2);
\fill (0,0) circle (2pt);
\node[ur] at (  45:1.2) {$d$};
\node[ul] at ( 135:1.2) {$c$};
\node[lr] at ( -45:1.2) {$b$};
\node[ll] at (-135:1.2) {$a$};
\end{tikzpicture}
\\
$D_+$ & $D_-$ & $D_r$ & $D_s$
\end{tabular}
\caption{Positive crossing, negative crossing, oriented smoothing and singular
crossing.  The dotted line connecting the two arcs of the oriented smoothing
illustrates that we consider both arcs to be on the same component of the
smoothing.}
\label{figure-smoothings}
\end{figure}

Our definition of \KR homology closely follows Rasmussen~\cite{rasmussen}, whose
definitions we amend slightly for technical reasons.  We also restrict
ourselves to connected diagrams.  To any diagram of a (possibly singular)
oriented tangle, which we allow to contain any of the diagrams depicted in
\autoref{figure-smoothings} as
subdiagrams, Rasmussen defines two rings, which depend only on the underlying
$4$-valent graph obtained by replacing all of those diagrams by a vertex.
The \emph{edge ring} $R(D)$ is the polynomial ring over $\Q$ generated
by variables $x_e$, where $e$ runs over all edges of the diagram, subject to a
relation of the form $x_a+x_b-x_c-x_d$ for each vertex of the underlying $4$-valent
graph.  By setting $\deg x_e = 2$ for each edge $e$ of $D$, $R(D)$ becomes a graded
ring. The \emph{external ring} $R_e(D)$ is the subring of $R(D)$ generated
by the variables associated to the endpoints of $D$. Lemma 2.5 in \cite{rasmussen}
shows that if we associate the variables $x_i$ ($i \in \{1,2,\dotsc, k\}$) to the
incoming edges of $D$ and $y_i$ to the outgoing edges, then
$R_e(D) \cong \Q[x_1, ..., x_n, y_1, ..., y_n]\slash\left(\sum_i y_i - \sum_i x_i\right)$.

Fix a polynomial $p \in \Q[x]$. If $p$ is not homogeneous, we will disregard
$q$-gradings below.  To each tangle diagram $D$, we will associate a complex
$C_p(D)$ of matrix factorizations over $R(D)$, which we consider to be an object of
the category $\K(\hmf_w(R_e(D)))$, where $\K(\CC)$ denotes the homotopy
category of bounded complexes over the additive category $\CC$
and $w = \sum_i p(y_i) - \sum_i p(x_i)$ where $x_i$ and $y_i$ are associated to the
incoming and outgoing edges as above.
$C_p(D)$ is first defined on the diagrams shown in \autoref{figure-smoothings}.
In each case $R \ce R(D) = R_e(D) = \Q[x_a,x_b,x_c,x_d]\slash(x_a+x_b=x_c+x_d)$. We
set
\begin{align*}
C_p(D_r) &= K(x_c-x_a;*)\mfshift = K(*; x_c-x_a)\{n-1\}\text{,} \\
C_p(D_s) &= K(*; x_cx_d-x_ax_b)\{-1\}\text{,} \\
C_p(D_+) &= K(*; x_cx_d-x_ax_b)[-1] \xrightarrow{d_+} K(*; x_c-x_a)\text{ and} \\
C_p(D_-) &= K(*; x_c-x_a) \xrightarrow{d_-} K(*; x_cx_d-x_ax_b)\{-2\}[1]\text{.}
\end{align*}
Here $*$ is of course determined by the potential in each case; we postpone
the definitions of $d_+$ and $d_-$ until we need them in \autoref{reflection}.

This definition is extended to arbitrary tangle diagrams by the formula
\begin{equation}\label{tensor}
C_p(D) = \bigotimes_i C_p(D_i) \otimes_{R(D_i)} R(D) \text{,}
\end{equation}
where $D_i$ runs over all crossings of $D$ and
the big tensor product is taken over $R(D)$. As indicated above, we usually view
$C_p(D)$ as a matrix factorization over the smaller ring $R_e(D)$.

Rasmussen shows (Lemma 2.8 in \cite{rasmussen})
\begin{proposition}\label{gluing}
If $D$ is obtained from $D_1$ and $D_2$ by taking their disjoint union
and identifying external edges labeled $x_1, \dotsc, x_k$ in both diagrams,
then
\[R(D) \cong R(D_1) \otimes_{\Q[x_1, \dotsc, x_k]} R(D_2) \text{ and}\]
\[C_p(D) \cong C_p(D_1) \otimes_{\Q[x_1, \dotsc, x_k]} C_p(D_2) \text{.}\]
\end{proposition}

To define reduced \KR homology of a link with respect to a marked component, we
pick an edge on the marked component, which we label by $x$.  We view $C_p(D)$
as a matrix factorization over $\Q[x]$, i.e.\ as an object of $\K(\hmf_0(\Q[x]))$.
Alternatively, we may consider the diagram $D^\circ$ obtained from $D$ by
cutting it open at the marked edge.  Let $x$ and $y$ be the labels of the
incoming and outgoing edge of $D^\circ$, respectively.  Then $C_p(D^\circ)$ is
a complex of matrix factorization with potential $p(y)-p(x) = 0$ over the ring
$R_e(D^\circ) = \Q[x,y]\slash(y-x) \cong \Q[x]$ and
$C_p(D) \cong C_p(D^\circ)$ as objects of $\K(\hmf_0(\Q[x]))$.

We define reduced \KR homology in two steps.  We first define
the unreduced complex $\widetilde C_p(D^\circ)$ by tensoring with $K(p'(x);0)$.
Then we define the reduced complex $\widehat C_p(D^\circ)$ as the mapping cone
$\Cone(x: C_p(\widetilde D^\circ)\{2\} \to C_p(\widetilde D^\circ))$ (we use 
$\widehat \cdot$ rather than $\bar \cdot$ in order to avoid confusion with
the involution $\bar \cdot$ to be defined later).

We now consider $\widehat C_p(D^\circ)$ as an object of $\K(\hmf_0(\Q))$. Since
we are working over a field
and matrix factorizations with potential $0$ are simply $\Z_2$-graded chain
complexes, the category $\hmf_0(\Q)$ is equivalent to the category of
$\Z_2 \oplus \Z$-graded $\Q$-vector spaces by \autoref{homotopy-2} below.
Hence the category $\K(\hmf_0(\Q)$ is equivalent to the category
$\K(\Z_2 \oplus \Z\text{-graded $\Q$-vector spaces})$, which in turn is
equivalent to the category of $\Z_2 \oplus \Z \oplus \Z$-graded
$\Q$-vector spaces (bounded with respect to the second $\Z$ summand) by \autoref{homotopy-0}.
Reduced \KR homology is the image of $\widehat C_p(D^\circ)\{(n-1)w\}$ under this
equivalence of categories, where $w$ is the writhe of $D$.

\begin{proposition}\label{equivalence}
The definition of reduced $\sln$ homology above is equivalent to Khovanov and
Rozansky's original definition in \cite{khr1}.
\end{proposition}
\begin{proof}
Rasmussen did most of the work for us in his proof of Proposition 3.12 in
\cite{rasmussen}.  Using the notation $d_{tot}$ for the (inner) matrix
factorization differential and $d_v$ for the (outer) differential of a complex
of matrix factorizations, he shows in Lemma 3.11 that the original definition
in \cite{khr1} is equivalent to $H^*(H/(x))$, where $H$ is the chain complex
whose underlying $\Q$-vector space is $H^*(C_p(D); d_{tot})$ and whose
differential is $d_v^*$.  Our definition above is equivalent to
$H^*(\Cone(x: H \to H)))$. Since $\Cone(x: H \to H)$ has a natural double complex
structure, it induces a spectral sequence that converges to
$H^*(\Cone(x: H \to H)))$. The $E^1$ page of the spectral sequence is $H/(x)$,
so the $E^2$ page is $H^*(H/(x))$.  The fact that the spectral sequence
collapses at the $E^2$ page implies that the two definitions are equivalent.
It is easy to see that gradings match as well, as Rasmussen explains at the
end of his proof of Proposition 3.12.
\end{proof}
\begin{remark}
We could have more straightforwardly defined reduced homology as the
homology of $\Cone(x: C_p(D) \to C_p(D))$.  However, this would have
required us to show that $C_p(D)$ is torsion-free as a $\Q[x]$-module
in order to establish equivalence of definitions, which follows from
\autoref{reduce} below only if $D$ is a braid diagram.
\end{remark}

The following two propositions are well-known in the finitely generated case.
We verify that proof carries over to the infinitely generated setting.
\begin{proposition}\label{homotopy-0}
Let
$C = \dots \xrightarrow{d^{k-1}} C^k \xrightarrow{d^k} C^{k+1} \xrightarrow{d^{k+1}} \dots $
be a chain complex over $\Q$ (with not necessarily finitely generated
chain groups).  Then $C$ is homotopy equivalent to a complex with zero
differential (its cohomology).
\end{proposition}
\begin{proof}
As usual, let $Z^k \ce \ker(d^k)$ and $B^{k+1} \ce \im(d^k)$.  Since vector
spaces are free as modules, the short exact sequences
$0 \rightarrow Z^k \rightarrow C^k \xrightarrow{d^k} B^{k+1} \rightarrow 0$
and
$0 \rightarrow B^k \rightarrow Z^k \rightarrow H^k(C) \rightarrow 0$
split.  It is easy to check that with respect to the decomposition
$C^k \cong Z^k \oplus B^{k+1} \cong B^k \oplus H^k(C) \oplus B^{k+1}$,
$C$ decomposes as a direct sum of chain complexes 
$0 \rightarrow H^k(C) \rightarrow 0$
and
$0 \rightarrow B^k \xrightarrow{\id} B^k \rightarrow 0$.
The Proposition now follows from the fact that the latter chain complex is homotopy
equivalent to the zero complex.
\end{proof}

\begin{proposition}\label{homotopy-2}
Any $\Z_2$-graded chain complex
$\mf{C^1}{C^0}{d^1}{d^0}$
is homotopy equivalent to a $\Z_2$-graded chain complex with zero differential.
\end{proposition}
\begin{proof}
Arguing as in the proof of the previous Proposition, we may decompose $C$ as
a direct sum of 
$\mf{H^1(C)}{0}{}{}$, $\mf{0}{H^0(C)}{}{}$, $\mf{B^1}{B^1}{0}{\id}$ and $\mf{B^0}{B^0}{\id}{0}$,
where the latter two complexes are homotopy equivalent to zero complexes.
\end{proof}

\section{Topological considerations}\label{section-topology}

In this section, we show that we may assume that the inner tangle is presented
in a specific form.

\begin{samepage}
\begin{definition}
We say that a 2-tangle is in braid form if it is represented in the following way, where
the rectangle represents an open braid.
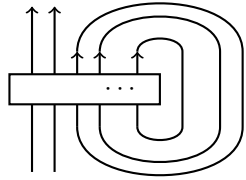
\begin{center}
\begin{tikzpicture}[thick]
\foreach \x in {0.3, 0.6} {\draw[->] (\x,-0.6)--(\x,1.6);}
\foreach \x in {0.9, 1.2, 1.7} {\draw[->] (\x,0)--(\x,1);}
\draw[fill=white](0,.3)rectangle(2,.7);
\draw (0.9,1) to[out=90,in=90] (3.1,1) to[out=-90,in=90] (3.1,0) to[out=-90,in=-90] (0.9,0);
\draw (1.2,1) to[out=90,in=90] (2.8,1) to[out=-90,in=90] (2.8,0) to[out=-90,in=-90] (1.2,0);
\draw (1.7,1) to[out=90,in=90] (2.3,1) to[out=-90,in=90] (2.3,0) to[out=-90,in=-90] (1.7,0);
\node at (1.5,.5){$\cdots$};
\end{tikzpicture}
\end{center}
\end{definition}
\end{samepage}
\begin{theorem}\label{braid-form}
Let $L$ be an oriented link and $L'$ be a mutant of $L$ obtained by
positive mutation.  Then the mutation can be represented
on a diagram whose inner tangle is given in braid form by a transformation
of type $R_y$.
\end{theorem}

The following two lemmas immediately imply the Theorem.

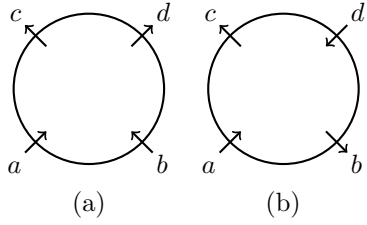
\begin{figure}
\centering
\begin{tabular}{cc}
\begin{tikzpicture}[thick]
\draw (0,0) circle (1cm);
\draw[->] (  45:0.8)--(  45:1.2) node[ur] {$d$};
\draw[->] ( 135:0.8)--( 135:1.2) node[ul] {$c$};
\draw[<-] ( -45:0.8)--( -45:1.2) node[lr] {$b$};
\draw[<-] (-135:0.8)--(-135:1.2) node[ll] {$a$};
\end{tikzpicture}
&
\begin{tikzpicture}[thick]
\draw (0,0) circle (1cm);
\draw[<-] (  45:0.8)--(  45:1.2) node[ur] {$d$};
\draw[->] ( 135:0.8)--( 135:1.2) node[ul]  {$c$};
\draw[->] ( -45:0.8)--( -45:1.2) node[lr] {$b$};
\draw[<-] (-135:0.8)--(-135:1.2) node[ll]  {$a$};
\end{tikzpicture}
\\
(a) & (b)
\end{tabular}
\caption{The two possible orientations of the endpoints of inner tangle}
\label{figure-orientations}
\end{figure}

\begin{lemma}\label{lemma-orientations}
We may assume that the endpoints of the inner tangle are oriented as in
\autoref{figure-orientations}(a) and that the transformation of the inner
tangle is of type $R_y$.
\end{lemma}

\begin{lemma}\label{lemma-braid-form}
Any 2-tangle with endpoints oriented as in \autoref{figure-orientations}(a)
can be represented by a diagram in braid form.
\end{lemma}

\begin{proof}[Proof (of \autoref{lemma-orientations})]
If the tangle has two adjacent endpoints with the same orientation, it is
isotopic to a tangle with endpoints as depicted in \autoref{figure-orientations}(a)
and the only positive mutation is of type $R_y$.  Otherwise we are in case (b)
of \autoref{figure-orientations}, where the only positive mutation is of type $R_z$.
But we can realize this type of mutation by $R_y$-mutation on a tangle of type
(a), as illustrated in \autoref{figure-mutation-type}.
\end{proof}

\begin{figure}
\begin{tabular}{ccc}
\begin{tikzpicture}[baseline, thick, scale=0.6]
\path[use as bounding box] (-1.3,-1.3) rectangle (1.3,1.3);
\draw (0,0) circle (1cm);
\draw[<-] (  45:0.7)--(  45:1.3);
\draw[->] ( 135:0.7)--( 135:1.3);
\draw[->] ( -45:0.7)--( -45:1.3);
\draw[<-] (-135:0.7)--(-135:1.3);
\node at (0,0) {$R$};
\end{tikzpicture}
&
$\approx$
&
\begin{tikzpicture}[baseline, thick, scale=0.6]
\path[use as bounding box] (-1.8,-2.3) rectangle (4.3,2.3);
\coordinate (d) at (  45:0.7);
\coordinate (c) at ( 135:0.7);
\coordinate (b) at ( -45:0.7);
\coordinate (a) at (-135:0.7);

\coordinate (dd) at ($ (  45:2.3) + (0.5,0) $);
\coordinate (cc) at ($ ( 135:2.3) + (0.5,0) $);
\coordinate (bb) at ($ ( -45:2.3) + (0.5,0) $);
\coordinate (aa) at ($ (-135:2.3) + (0.5,0) $);

\coordinate (ddd) at ($ (  45:2.3) + (2.5,0) $);
\coordinate (bbb) at ($ ( -45:2.3) + (2.5,0) $);

\draw (0,0) circle (1cm);
\node at (0,0) {$R$};
\draw (0.5,0) circle (2cm);
\draw[<-] (a) to[out=-135,in=  45] (aa);
\draw[->] (c) to[out= 135,in= -45] (cc);
\draw     (b) to[out= -45,in=-135] (dd);
\fill[white] (1.1,-0.15) rectangle (1.6,0.15);
\draw[<-] (d) to[out=  45,in= 135] (bb);

\draw[->] (dd) to[out= 45,in= 135] (bbb);
\fill[white] (3,-0.15) rectangle (3.5,0.15);
\draw     (bb) to[out=-45,in=-135] (ddd);
\end{tikzpicture}
\\
\begin{tikzpicture}[baseline, thick, scale=0.6]
\path[use as bounding box] (-1.3,-1.3) rectangle (1.3,1.3);
\begin{scope}[transform canvas={rotate around={180:(0,0)}}]
\draw (0,0) circle (1cm);
\draw[<-] (  45:0.7)--(  45:1.3);
\draw[->] ( 135:0.7)--( 135:1.3);
\draw[->] ( -45:0.7)--( -45:1.3);
\draw[<-] (-135:0.7)--(-135:1.3);
\node at (0,0) {$R$};
\end{scope}
\end{tikzpicture}
&
$\approx$
&
\begin{tikzpicture}[baseline, thick, scale=0.6]
\path[use as bounding box] (-1.8,-2.3) rectangle (4.3,2.3);
\coordinate (d) at (  45:0.7);
\coordinate (c) at ( 135:0.7);
\coordinate (b) at ( -45:0.7);
\coordinate (a) at (-135:0.7);

\coordinate (dd) at ($ (  45:2.3) + (0.5,0) $);
\coordinate (cc) at ($ ( 135:2.3) + (0.5,0) $);
\coordinate (bb) at ($ ( -45:2.3) + (0.5,0) $);
\coordinate (aa) at ($ (-135:2.3) + (0.5,0) $);

\coordinate (ddd) at ($ (  45:2.3) + (2.5,0) $);
\coordinate (bbb) at ($ ( -45:2.3) + (2.5,0) $);

\begin{scope}[transform canvas={xshift=0.6cm,xscale=-1}]
\draw (0,0) circle (1cm);
\node at (0,0) {$R$};
\draw (0.5,0) circle (2cm);
\draw[<-] (a) to[out=-135,in=  45] (aa);
\draw     (c) to[out= 135,in= -45] (cc);
\draw[<-] (d) to[out=  45,in= 135] (bb);
\fill[white] (1.1,-0.15) rectangle (1.6,0.15);
\draw[->] (b) to[out= -45,in=-135] (dd);
\end{scope}

\draw[->] (dd) to[out= 45,in= 135] (bbb);
\fill[white] (3,-0.15) rectangle (3.5,0.15);
\draw     (bb) to[out=-45,in=-135] (ddd);
\end{tikzpicture}
\end{tabular}
\caption{$R_z$ mutation on a tangle of type (b) is equivalent to $R_y$ mutation on a tangle of type (a)}
\label{figure-mutation-type}
\end{figure}
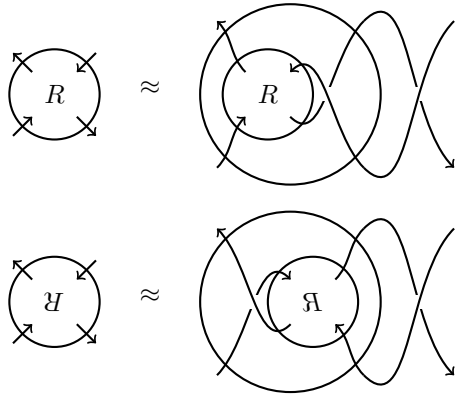

\begin{proof}[Proof (of \autoref{lemma-braid-form})]
The proof uses a slight variation of the Yamada-Vogel~\cites{yamada,vogel}
algorithm to prove an analog of Alexander's Theorem for 2-tangles.  We follow
Birman and Brendle~\cite{birman-brendle}.

\begin{figure}
\centering
\begin{tabular}{cc}
\begin{tikzpicture}[thick, baseline=0cm]
\draw ( 45:-1) to[out=-135,in=150] (0,-1.5) to[out= -30,in=-90] ( 1.7,0);
\draw[<-] (1.7,0) to[out=90,in= 30] (0, 1.5) to[out=-150,in=135] (135:1);
\draw[upper] (135:-1) to[out= -45,in= 30] (0,-1.5) to[out=-150,in=-90] (-1.7,0);
\draw[upper,<-] (-1.7,0) to[out=90,in=150] (0, 1.5) to[out= -30,in= 45] ( 45:1);
\draw (0,0) circle (1cm);
\node[cr] at (  45:1) {$d$};
\node[cl] at ( 135:1) {$c$};
\node[cr] at ( -45:1) {$b$};
\node[cl] at (-135:1) {$a$};
\node[cl] at (-1.7,0)   {$\alpha$};
\node[cr] at ( 1.7,0)   {$\beta$};
\end{tikzpicture}
&
\begin{tikzpicture}[thick, baseline=0cm]
\draw (0,0) circle (1cm);
\draw[looseness=1.5   ] ( 45:-1) to[out=-135,in=-90] (-1.6,0);
\draw[looseness=1.5,<-] (-1.6,0) to[out=90,in=135] (135:1);
\draw[looseness=1.5   ] (135:-1) to[out=-45,in=-90] (1.6,0);
\draw[looseness=1.5,<-] (1.6,0)  to[out=90,in=45] ( 45:1);
\draw[dotted] ( 135:1.2) to[out= 45,in= 135] ( 45:1.2);
\draw[dotted] (-135:1.2) to[out=-45,in=-135] (-45:1.2);
\node[cr] at (  45:1) {$d$};
\node[cl] at ( 135:1) {$c$};
\node[cr] at ( -45:1) {$b$};
\node[cl] at (-135:1) {$a$};
\node[cl] at (-1.6,0)   {$\alpha$};
\node[cr] at ( 1.6,0)   {$\beta$};
\end{tikzpicture}
\\ (a) & (b)
\end{tabular}
\caption{A closure of the tangle and its Seifert picture}
\label{figure-yamada-vogel}
\end{figure}
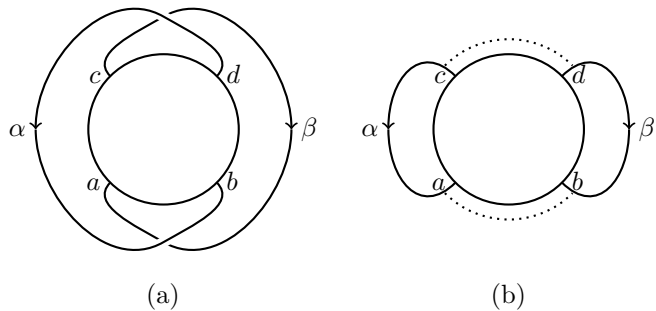

\begin{figure}
\centering
\begin{tabular}{cc}
\begin{tikzpicture}[baseline, thick]
\foreach \x in {0.2 ,0.7, 1, 1.9, 2.2, 2.7} {\draw[->] (\x,0)--(\x,1);}
\foreach \x in {1.3, 1.6} {\draw[->] (\x,-0.6)--(\x,1.6);}
\draw[fill=white](-0.1,.4)rectangle(3,.7);
\draw (1.9,1) to[out=90,in=90] (4.1,1) to[out=-90,in=90] (4.1,0) to[out=-90,in=-90] (1.9,0);
\draw (2.2,1) to[out=90,in=90] (3.8,1) to[out=-90,in=90] (3.8,0) to[out=-90,in=-90] (2.2,0);
\draw (2.7,1) to[out=90,in=90] (3.3,1) to[out=-90,in=90] (3.3,0) to[out=-90,in=-90] (2.7,0);

\draw (  1,1) to[out=90,in=90] (-1.2,1) to[out=-90,in=90] (-1.2,0) to[out=-90,in=-90] (  1,0);
\draw (0.7,1) to[out=90,in=90] (-0.9,1) to[out=-90,in=90] (-0.9,0) to[out=-90,in=-90] (0.7,0);
\draw (0.2,1) to[out=90,in=90] (-0.4,1) to[out=-90,in=90] (-0.4,0) to[out=-90,in=-90] (0.2,0);
\node at (2.5,.5){$\cdots$};
\node at (0.5,.5){$\cdots$};
\end{tikzpicture}
&
\begin{tikzpicture}[baseline, thick]
\foreach \x in {0.2 ,0.7, 1, 1.9, 2.2, 2.7} {\draw[->] (\x,0)--(\x,1);}
\draw (1.9,1) to[out=90,in=90] (4.1,1) to[out=-90,in=90] (4.1,0) to[out=-90,in=-90] (1.9,0);
\draw (2.2,1) to[out=90,in=90] (3.8,1) to[out=-90,in=90] (3.8,0) to[out=-90,in=-90] (2.2,0);
\draw (2.7,1) to[out=90,in=90] (3.3,1) to[out=-90,in=90] (3.3,0) to[out=-90,in=-90] (2.7,0);

\draw (  1,1) to[out=90,in=90] (4.4,1) to[out=-90,in=90] (4.4,0) to[out=-90,in=-90] (  1,0);
\draw (0.7,1) to[out=90,in=90] (4.7,1) to[out=-90,in=90] (4.7,0) to[out=-90,in=-90] (0.7,0);
\draw (0.2,1) to[out=90,in=90] (5.2,1) to[out=-90,in=90] (5.2,0) to[out=-90,in=-90] (0.2,0);

\draw[->,upper,looseness=0.6]
  (0.2,-1.5)--(0.2,-1.4) to[out=90,in=-90] (1.3,0.1)--(1.3,0.9) to[out=90,in=-90] (0.2,2.4)--(0.2,2.5);
\draw[->,upper,looseness=0.6]
  (0.5,-1.5)             to[out=90,in=-90] (1.6,  0)--(1.6,  1) to[out=90,in=-90] (0.5,2.5);

\draw[fill=white](-0.1,.4)rectangle(3,.7);
\node at (2.5,.5){$\cdots$};
\node at (0.5,.5){$\cdots$};
\end{tikzpicture}
\end{tabular}
\caption{Transforming the diagram into braid form}
\label{figure-braid-form}
\end{figure}
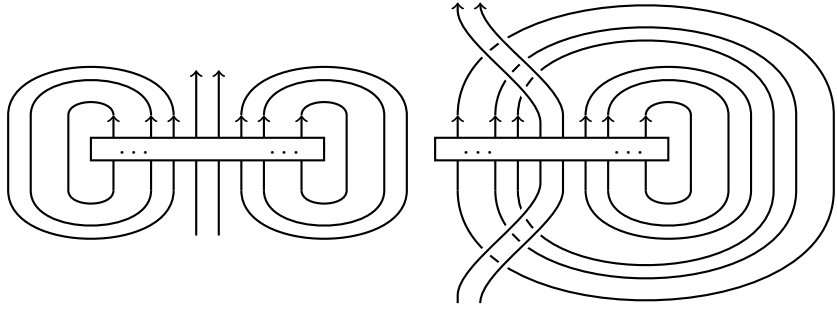

Close the tangle by two arcs $\alpha$ from $c$ to $a$ and $\beta$ from $d$ to
$b$ as in \autoref{figure-yamada-vogel}(a).  The algorithm works by repeatedly
performing a Reidemeister II move in a small neighborhood of a so-called
\emph{reducing arc}.  The algorithm is performed on the Seifert picture of
the link diagram, which is depicted in \autoref{figure-yamada-vogel}(b).  A reducing
arc is an arc connecting an incoherently oriented pair of Seifert circles that intersects
the Seifert picture only at its endpoints.  Since the Seifert circles that
$\alpha$ and $\beta$ belong to are coherently oriented, the
the unbounded region of the Seifert picture in \autoref{figure-yamada-vogel}(b)
cannot contain a reducing arc.  Hence we may push the reducing arc into
the circle.  The algorithm now gives us a tangle diagram whose Seifert circles
and Seifert arcs (from $a$ to $c$ and from $b$ to $d$) are coherently oriented.
This implies that all Seifert circles lie nested inside each other to
the left of the left arc and to the right of the right arc, in other words
it can be represented by a diagram in the form illustrated on the left
of \autoref{figure-braid-form}.  But this can be easily transformed into
braid form, as seen on the right of \autoref{figure-braid-form}.

\end{proof}

\section{Behavior of the \KR chain complex under reflection}\label{section-reflection}
\begin{lemma}\label{reflection}
Let $D$ be an oriented (possibly singular) tangle diagram and $\bar D$ be the reflection of $D$.
Label the endpoints of $D$ by $e_0, e_1, \dotsc, e_{2k-1}$, and the
corresponding endpoints of $\bar D$ by $e'_0, e'_1, \dotsc, e'_{2k-1}$.
Then $C_p(\bar D)$ = $\phi(C_p(D))$, where $\phi: R(D) \to R(\bar D)$ is the ring
homomorphism given by $\phi(x_{e_i}) = - x_{e'_i}$.
\end{lemma}
\begin{proof}
If $D$ is one of the diagrams shown in \autoref{figure-smoothings}, $C_p(D)$ is one of
the following complexes of matrix factorizations.

\begin{center}
\begin{tabular}{cc}
$C_p(D_+)$:
&
$C_p(D_-)$:
\\
$\xymatrix{
R\{1-n\} \ar@<0.4ex>[rrr]^-{x_c-x_a} &&& R \ar@<0.4ex>[lll]^-{\frac{w}{x_c-x_a}} \\
R\{3-n\} \ar@<0.4ex>[rrr]^-{(x_c-x_a)(x_c-x_b)} \ar[u]^{x_c-x_b} &&& R \ar@<0.4ex>[lll]^-{\frac{w}{(x_c-x_a)(x_c-x_b)}} \ar[u]_{1}
}$
&
$\xymatrix{
R\{1-n\} \ar@<0.4ex>[rrr]^{(x_c-x_a)(x_c-x_b)} &&& R\{-2\} \ar@<0.4ex>[lll]^{\frac{w}{(x_c-x_a)(x_c-x_b)}} \\
R\{1-n\} \ar@<0.4ex>[rrr]^{x_c-x_a} \ar[u]^{1} &&& R \ar@<0.4ex>[lll]^{\frac{w}{x_c-x_a}} \ar[u]_{x_c-x_b}
}$
\\
\\
$C_p(D_r)$:
&
$C_p(D_s)$:
\\
$\xymatrix{
\qquad R \ar@<0.4ex>[rr]^-{x_c-x_a} && R\{n-1\} \ar@<0.4ex>[ll]^-{\frac{w}{x_c-x_a}}
}$
&
$\xymatrix{
R\{2-n\} \ar@<0.4ex>[rrr]^-{(x_c-x_a)(x_c-x_b)} &&& R\{-1\} \ar@<0.4ex>[lll]^-{\frac{w}{(x_c-x_a)(x_c-x_b)}}
}$
\end{tabular}
\end{center}

Note that $\phi(x_c - x_a) = -x_{d'}+x_{b'} = x_{c'}-x_{a'}$, $\phi(x_c-x_b) = -x_{d'}+x_{a'} = x_{c'}-x_{b'}$
and 
$\phi(w) = \phi \left(\sum_k a_{2k} (x_c^{2k}+x_d^{2k} - x_a^{2k}-x_b^{2k})\right)
         = \sum_k a_{2k} (x_{c'}^{2k}+x_{d'}^{2k} - x_{a'}^{2k}-x_{b'}^{2k}) = w'$.
Hence all maps in the above diagrams are mapped by $\phi$ to the same maps with
$x_{a'}$, $x_{b'}$, $x_{c'}$ and $x_{d'}$ in place of $x_a$, $x_b$, $x_c$ and
$x_d$, respectively, that is $\phi$ maps $C_p(D)$ to $C_p(\bar D)$.

The general case follows from (\ref{tensor}): It is clear that by taking
the internal edges of $D$ into consideration, we can extend $\phi$ to
an isomorphism between $R(D)$ and $R(\bar D)$.  Hence we get isomorphisms
$C_p(D_i) \otimes_{R(D_i)} R(D) \cong C_p(\bar D_i) \otimes_{R(\bar D_i)} R(\bar D)$,
which in turn induce an isomorphism $C_p(D) \cong C_p(\bar D)$.
\end{proof}
In light of the Lemma, we will simply denote the homomorphism $\phi$ by $\bar \cdot$.

\section{\KR Homology of 2-tangles}\label{section-2-tangles}

In this section, we investigate the \KR homology of 2-tangles in braid form.
Denote the variables corresponding to the endpoints $a$, $b$, $c$ and $d$ of
the tangle by $x_a$, $x_b$, $x_c$ and $x_d$, respectively.  The complex
associated to such a tangle is a complex of graded matrix factorizations
over the ring $R = \Q[x_a,x_b,x_c,x_d]\slash(x_a+x_b=x_c+x_d)$ with potential
$w = p(x_c)+p(x_d)-p(x_a)-p(x_b)$.  Let $\hmf_2$ denote the full subcategory
of $\hmf_w(R)$ whose objects are direct sums of shifts of $C_p(D_r)$ and $C_p(D_s)$.

\begin{theorem}\label{\hmf2}
Let $D$ a connected diagram of a 2-tangle in braid-form.
Then $C_p(D)$ is isomorphic in $\K(\hmf_w(R))$ to an object of $\K(\hmf_2)$.
\end{theorem}

Before proving the theorem, we need to recall an important tool
for dealing with matrix factorizations: `excluding a variable'.
We quote Theorem 2.2 from \cite{khr3}.

\begin{theorem}\label{excluding}
Let R be a graded polynomial ring over $\Q$ and
and $u, v \in R[y]$ two polynomials, with $b$
being monic in $y$.  Furthermore, let $\bar w \in R$
and $M$ be a graded matrix factorization over $R[y]$ with potential 
$w = \bar w - uv$.  Then $M/(v)$ and $K(u;v) \otimes M$ are
isomorphic as objects of $hmf_w(R)$.  We say that
we \emph{exclude} the variable $y$ to obtain $M/(v)$ from $K(u;v) \otimes M$.
\end{theorem}
The Theorem is only stated for ungraded matrix factorization in \cite{khr3},
but it is trivial to check that the quotient map $K(u;v) \otimes M \to M/(v)$
constructed in the proof is of degree $0$. \qed

We will also need another well-known result about Koszul matrix factorizations;
this is, for example, the $n = 2$ special case of Theorem 2.1 in \cite{khr3}.
\begin{theorem}\label{stars2}
Let $R$ be a graded polynomial ring over $\Q$ and $v_1, v_2 \in R$ be
relatively prime.  Then any two Koszul matrix factorizations of the form
$\koszul{* & v_1 \\ * & v_2}$ with the same potential are isomorphic.
\end{theorem}

In the same spirit, we show that a matrix factorization that is almost
the direct sum of two order-two Koszul matrix factorizations can be
transformed into an honest direct sum.
\begin{theorem}\label{stars4}
Let $R$ be a graded polynomial ring and $\dot R = R\{\dot k\}$ and 
$\ddot R = R\{\ddot k\}$ be free $R$-modules of rank $1$, 
then any graded matrix factorization of the form
\[\mf%
{\dot R\{k_a\} \oplus \dot R\{k_b\} \oplus \ddot R\{k_c\} \oplus \ddot R\{k_d\}}%
{\dot R\{k_a+k_b\} \oplus \dot R \oplus \ddot R\{k_c+k_d\} \oplus \ddot R}%
{V}%
{U}%
\]
with 
\[
U = \begin{pmatrix} b & * & 0 & * \\ a &  * & 0 & * \\ 0 & * & d & * \\ 0 & * & c & * \end{pmatrix}
\hspace{0.5cm}\text{and}\hspace{0.5cm}
V = \begin{pmatrix} * & * & * & * \\ a & -b & 0 & 0 \\ * & * & * & * \\ 0 & 0 & c & -d \end{pmatrix}
\text{,}
\]
where $\gcd(a,b) = \gcd(c,d) = 1$ and $k_x = \deg x - \frac{\deg w}2$ for $x \in \{a,b,c,d\}$,
is isomorphic to a matrix factorization of the form
\[ \koszul{* & a \\ * & b}\!\{\dot k\} \oplus \koszul{* & c \\ * & d}\!\{\ddot k\} \]
\end{theorem}
\begin{proof}
Let 
\[
U = \begin{pmatrix} b & * & 0 & b_2 \\ a & * & 0 & a_2 \\ 0 & d_1 & d & * \\ 0 & c_1 & c & * \end{pmatrix}
\hspace{0.5cm}\text{and}\hspace{0.5cm}
V = \begin{pmatrix} * & * & c_2 & -d_2 \\ a & -b & 0 & 0 \\ a_1 & -b_1 & * & * \\ 0 & 0 & c & -d \end{pmatrix}
\text{.}
\]
Computing the lower left and the upper right quadrant of $UV = wI$, we see that
$\mmatrix{d_1}{d}{c_1}{c} \mmatrix{a}{-b}{a_1}{-b_1} = 0$ and
$\mmatrix{b}{b_2}{a}{a_2} \mmatrix{c_2}{-d_2}{c}{-d} = 0$.
Since $\gcd(a,b) = \gcd(c,d) = 1$, the rank of each of these matrices is
at least $1$, so none of them can have rank $2$. Hence
$0 = \mathrm{det} \mmatrix{a}{-b}{a_1}{-b_1} = -b_1a+a_1b$ and there
exists an $\alpha \in R$ such that $a_1 = \alpha_1 a$ and $b_1 = \alpha_1 b$.
Similarly, we can find an $\alpha_2 \in R$, such that
$a_2 = \alpha_2 a$ and $b_2 = \alpha_2 b$, as well as $\beta_i \in R$ ($i \in \{1,2\}$)
such that $c_i = \beta_i c$ and $d_i = \beta_i d$.  The fact that the two matrix
products above are $0$ implies that $\beta_i = -\alpha_i$.  We now perform a change
of basis,
\[\xymatrix{
\dot R\{k_a\} \oplus \dot R\{k_b\} \oplus \ddot R\{k_c\} \oplus \ddot R\{k_d\} \ar@<0.4ex>[r]^{V } \ar@{=}[d] &
\dot R\{k_a+k_b\} \oplus \dot R \oplus \ddot R\{k_c+k_d\} \oplus \ddot R \ar@<0.4ex>[l]^{U } \ar@<0.4ex>[d]^{X} \\
\dot R\{k_a\} \oplus \dot R\{k_b\} \oplus \ddot R\{k_c\} \oplus \ddot R\{k_d\} \ar@<0.4ex>[r]^{V'} &
\dot R\{k_a+k_b\} \oplus \dot R \oplus \ddot R\{k_c+k_d\} \oplus \ddot R \ar@<0.4ex>[l]^{U'} \ar@<0.4ex>[u]^{X\inv} 
}\]
where
\[
X = \begin{pmatrix} 1 &&& \alpha_2 \\ & 1 && \\ & -\alpha_1 & 1 & \\ &&& 1 \end{pmatrix}
\text{,}\hspace{0.2cm}
U' = \begin{pmatrix} b & * & 0 & 0 \\ a &  * & 0 & 0 \\ 0 & 0 & d & * \\ 0 & 0 & c & * \end{pmatrix}
\hspace{0.1cm}\text{and}\hspace{0.2cm}
V' = \begin{pmatrix} * & * & 0 & 0 \\ a & -b & 0 & 0 \\ 0 & 0 & * & * \\ 0 & 0 & c & -d \end{pmatrix}
\text{,}
\]
the lower row being exactly the desired direct sum of Koszul matrix factorizations.

We still need to verify that $C$ is of degree 0:  We have
$\deg \alpha_1 = \deg a_1 - \deg a
= (\frac{\deg w}2 + k_a + \ddot k - k_c - k_d - \dot k) - (k_a + \frac{\deg w}2)
= \ddot k - \dot k - k_c - k_d$ and
$\deg \alpha_2 = \deg a_2 - \deg a
= (\frac{\deg w}2 + \dot k - k_b - \ddot k) - (k_a + \frac{\deg w}2)
= \dot k - \ddot k - k_a - k_b$, which
implies $\deg(-\alpha_1: \dot R \to \ddot R\{k_c+k_d\}) = \deg(\alpha_2: \ddot R \to \dot R\{k_a+k_b\}) = 0$.
\end{proof}

The following proposition is an analog of Lemma 4.10 and Propositions 4.3--4.6
in \cite{rasmussen} and Lemma 3 and Propositions 4--7 in \cite{khr2}.  Unfortunately,
we cannot deduce it from any of the previous results:  The theory introduced in
\cite{khr1} is weaker than what we consider here (in the $\mysl2$ case,
this is the difference between Khovanov Homology and Bar-Natan's
universal variant \cite{bar-natan}).  We also cannot use the results in 
\cite{rasmussen}, which are only shown to hold up to a notion of
quasi-isomorphism.  However, the proofs in Rasmussen' paper can be
modified to apply to our situation.

\begin{figure}
\begin{tabular}{rccrccrc}
$D_A$ &
\begin{tikzpicture}[baseline=0.4cm,thick,scale=0.5]
\draw[->] (0,0) -- (0,2);
\node[cr] at (0,1) {$a$};
\end{tikzpicture}
&& $D_r$ &
\begin{tikzpicture}[baseline=0.4cm,thick,scale=0.5]
\draw[->] (0,0) node[b]{$a$} -- (0,2) node[t]{$c$};
\draw[->] (1,0) node[b]{$b$} -- (1,2) node[t]{$d$};
\draw[densely dotted] (0,1)--(1,1);
\end{tikzpicture}
&& $D_s$ &
\begin{tikzpicture}[baseline=0.4cm,thick,scale=0.5]
\draw[->] (0,0) node[b]{$a$} to[out=90,in=-90] (1,2) node[t]{$d$};
\draw[->] (1,0) node[b]{$b$} to[out=90,in=-90] (0,2) node[t]{$c$};
\fill (0.5,1) circle (3.5pt);
\end{tikzpicture}
\\ $D_O$ &
\begin{tikzpicture}[baseline=0.4cm,thick,scale=0.5]
\draw[->] (0,-0.3) node[b]{$a$} -- (0,2.3) node[t]{$c$};
\draw (1,0) -- (1,2) to[out=90,in=90] (2,2) -- (2,0) to[out=-90,in=-90] (1,0);
\node[right] at (2,1) {$b$};
\draw[densely dotted] (0,1)--(1,1);
\end{tikzpicture}
&& $D_I$ &
\begin{tikzpicture}[baseline=0.4cm,thick,scale=0.5]
\draw[->] (0,-0.3) node[b]{$a$} -- (0,0) to[out=90,in=-90] (1,2) to[out=90,in=90] (2,2) -- (2,0) to[out=-90,in=-90] (1,0) to[out=90,in=-90] (0,2)--(0,2.3) node[t]{$c$};
\node[right] at (2,1) {$b$};
\fill (0.5,1) circle (3.5pt);
\end{tikzpicture}
&& $D_{II}$ &
\begin{tikzpicture}[baseline=0.7cm,thick,scale=0.5]
\draw[->] (0,0) node[b]{$a$} to[out=90,in=-90] (1,1.6) node[cr]{$y$} to[out=90,in=-90] (0,3.2) node[t]{$c$};
\draw[->] (1,0) node[b]{$b$} to[out=90,in=-90] (0,1.6) node[cl]{$x$} to[out=90,in=-90] (1,3.2) node[t]{$d$};
\fill (0.5,0.8) circle (3.5pt);
\fill (0.5,2.4) circle (3.5pt);
\end{tikzpicture}
\\ $D_{IIIa}$ &
\begin{tikzpicture}[baseline=0.9cm,thick,scale=0.5]
\draw[->] (0,0) node[b]{$a$} -- (0,1) to[out=90,in=-90] (2,4) node[t]{$f$};
\draw[->] (1,0) node[b]{$b$} to[out=90,in=-90] (2,2) node[cr]{$y$} to[out=90,in=-90] (1,4) node[t]{$e$};
\draw[->] (2,0) node[b]{$c$} to[out=90,in=-90] (0,3) -- (0,4) node[t]{$d$};
\fill (1.54,1.05) circle (3.5pt);
\fill (0.42,2) circle (3.5pt);
\fill (1.54,2.95) circle (3.5pt);
\node at (0.8,2.7) {$x$};
\node at (0.8,1.25) {$z$};
\end{tikzpicture}
&& $D'_{IIIa}$ &
\begin{tikzpicture}[baseline=0.9cm,thick,scale=0.5]
\draw[->] (2,0) node[b]{$c$} -- (2,1) to[out=90,in=-90] (0,4) node[t]{$d$};
\draw[->] (1,0) node[b]{$b$} to[out=90,in=-90] (0,2) to[out=90,in=-90] (1,4) node[t]{$e$};
\draw[->] (0,0) node[b]{$a$} to[out=90,in=-90] (2,3) -- (2,4) node[t]{$f$};
\fill (0.46,1.05) circle (3.5pt);
\fill (1.58,2) circle (3.5pt);
\fill (0.46,2.95) circle (3.5pt);
\end{tikzpicture}
&& $D_{IV}$ &
\begin{tikzpicture}[baseline=0.9cm,thick,scale=0.5]
\draw[->] (0,0) node[b]{$a$} -- (0,4) node[t]{$d$};
\draw[->] (1,0) node[b]{$b$} to[out=90,in=-90] (2,2) to[out=90,in=-90] (1,4) node[t]{$e$};
\draw[->] (2,0) node[b]{$c$} to[out=90,in=-90] (1,2) to[out=90,in=-90] (2,4) node[t]{$f$};
\draw[densely dotted] (0,2)--(1,2);
\fill (1.5,1) circle (3.5pt);
\fill (1.5,3) circle (3.5pt);
\end{tikzpicture}
\\ $D_{IIIb}$ &
\begin{tikzpicture}[baseline=0.65cm,thick,scale=0.5]
\draw[->] (0,0) node[b]{$a$} -- (0,1) to[out=90,in=-90] (1,3) node[t]{$e$};
\draw[->] (1,0) node[b]{$b$} -- (1,1) to[out=90,in=-90] (0,3) node[t]{$d$};
\draw[->] (2,0) node[b]{$c$} -- (2,3) node[t]{$f$};
\draw[densely dotted] (1,0.8)--(2,0.8);
\fill (0.5,2) circle (3.5pt);
\end{tikzpicture}
&& $D'_{IIIb}$ &
\begin{tikzpicture}[baseline=0.65cm,thick,scale=0.5]
\draw[->] (0,0) node[b]{$a$} -- (0,3) node[t]{$d$};
\draw[->] (1,0) node[b]{$b$} -- (1,1) to[out=90,in=-90] (2,3) node[t]{$f$};
\draw[->] (2,0) node[b]{$c$} -- (2,1) to[out=90,in=-90] (1,3) node[t]{$e$};
\draw[densely dotted] (0,0.8)--(1,0.8);
\fill (1.5,2) circle (3.5pt);
\end{tikzpicture}
\end{tabular}
\caption{Singular braid diagrams}
\label{figure-braid-diagrams}
\end{figure}
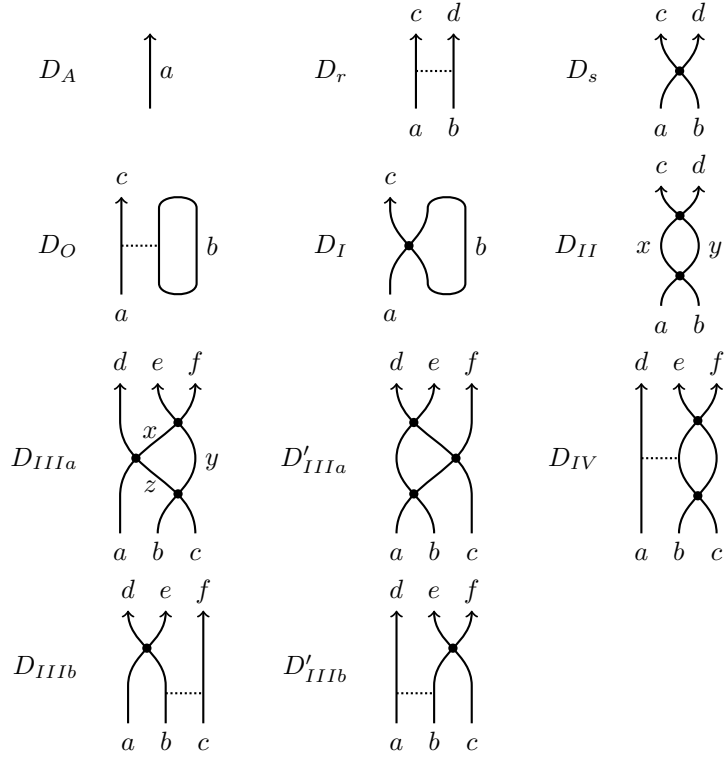

\pagebreak
\begin{proposition}\label{reduce}
The following isomorphisms hold in the homotopy category of matrix factorizations over the
external ring corresponding to the diagrams.
\begin{itemize}
\item[(a)] Let $D$ be a diagram of a fully resolved tangle, and $D'$ be a diagram
obtained from $D$ by replacing a smoothing of type $D_r$ (See \autoref{figure-smoothings})
by a pair of arcs without increasing the number of components.  Then $C_p(D) \cong C_p(D')$.
\item[(b)] Up to grading shifts, $C_p(D_O)$ is isomorphic to a direct sum of $n$ copies of $C_p(D_A)$.
\item[(c)] Up to grading shifts, $C_p(D_I)$ is isomorphic to a direct sum of $n-1$ copies of $C_p(D_A)$.
\item[(d)] Up to grading shifts, $C_p(D_{II}) \cong C_p(D_s) \oplus C_p(D_s)$.
\item[(e)] Up to grading shifts, $C_p(D_{IIIa}) \oplus C_p(D_{IIIb}) \cong C_p(D'_{IIIa}) \oplus C_p(D'_{IIIb})$.
\item[(f)] Up to grading shifts, $C_p(D_{IV}) \cong C_p(D'_{IIIb}) \oplus C_p(D'_{IIIb})$.
\end{itemize}
\end{proposition}
\begin{proof}
(a)
Since $D$ and $D'$ are connected, their external rings $R_e(D)$ and $R_e(D')$
are identical.  Since $R(D') = R(D)\slash(x_a=x_c)$ and $R(D) \cong R(D')[x]$
by Lemma 2.4 in \cite{rasmussen}, $x_a$ and $x_c$ are different elements of
$R(D)$.  If $x_a$ and $x_c$ were both linear combinations of external edges,
then their difference $x_c - x_a$ would be a linear combination of external
edges as well.  But $x_c - x_a \neq 0 \in R(D)$ and $x_c - x_a = 0 \in R(D')$,
which contradicts $R_e(D) = R_e(D')$.  Assume w.l.o.g. that $x_c$ is not a linear
combination of external edges.  Since $K(*; x_c - x_a)$ appears as a factor
of $C_p(D)$, we may exclude the variable $x_c$ to obtain $C_p(D')$.
\bigskip

(b) We have $R(D_O) = \Q[x_a,x_b,x_c]\slash(x_c-x_a)$, $R_e(D_O) = \Q[x_a,x_c]\slash(x_c-x_a)$ and
$R(D_A) = \Q[x_a]$, hence
$C_p(D_O)
= K\left(x_c-x_a; \frac{p(x_c)+p(x_b)-p(x_a)-p(x_b)}{x_c-x_a}\right)\mfshift
= K(0; p'(x_a) - p'(x_b))\mfshift
$.
Excluding the variable $x_b$, we obtain
\[C_p(D_O) \cong \bigoplus_{i=0}^{n-1} C_p(D_A)\mfshift\{2i\}\text{.}\]
\bigskip

(c) As in part (b), we have $R(D_s) = \Q[x_a,x_b,x_c]\slash(x_c-x_a)$, $R_e(D_s) = \Q[x_a,x_c]\slash(x_c-x_a)$ and
$R(D_A) = \Q[x_a]$, so
\begin{align*}
C_p(D_I) 
&= K\left(\tfrac{p(x_c)+p(x_b)-p(x_a)-p(x_b)}{(x_c-x_a)(x_c-x_b)}; (x_c-x_a)(x_c-x_b)\right)\{-1\} \\
&= K\left(\tfrac{p'(x_a)-p'(x_b)}{x_a-x_b}; 0\right)\{-1\} \\
&= K\left(0; \tfrac{p'(x_a)-p'(x_b)}{x_a-x_b}\right)\mfshift\{2-n\}
\text{,}
\end{align*}
so once again we may exclude $x_b$ to get
\[C_p(D_I) \cong \bigoplus_{i=0}^{n-2} C_p(D_A)\mfshift\{2-n+2i\}\text{.}\]
\bigskip

(d)
Choose labels as in \autoref{figure-braid-diagrams} and set $x \ce x_x$ and $y \ce
x_y$. As matrix factorizations over $R(D_{II})$,
\begin{align*}
C_p(D_{II})\{2\}
&=     \koszul{* & (x-x_a)(x-x_b) \\ * & (x_c-x)(x_c-y)} \\
&=     \koszul{* & (x-x_a)(x-x_b) \\ * & (x_c-x)(x-x_d)} \\
&\cong \koszul{* & (x-x_a)(x-x_b) \\ * & (x_c-x)(x-x_d) + (x-x_a)(x-x_b)} \\
&= \koszul{* & (x-x_a)(x-x_b) \\ * & (x_c-x_a)(x_c-x_b)}
\end{align*}
Let $R = R_e(D_{II}) = R_e(D_s)$.
Excluding the variable $x$, we get a matrix factorization $K(\alpha + \beta x; (x_c-x_a)(x_c-x_b))$
over the ring $R' = R[x]\slash(x^2 = (x_a+x_b)x - x_ax_b)$ whose potential
$(\alpha + \beta x)(x_c-x_a)(x_c-x_b)$ has to lie in $R$, hence $\beta = 0$.
As a graded $R$-module, $R' \cong R \oplus R\{2\}$, so
$C_p(D_{II}) \cong K_R(\alpha; (x_c-x_a)(x_c-x_b))\{-2\} \oplus K_R(\alpha; (x_c-x_a)(x_c-x_b))
\cong C_p(D_s)\{-1\} \oplus C_p(D_s)\{1\}$.
\bigskip

(e)
Choose labels as in \autoref{figure-braid-diagrams}, and set $x \ce x_x$, $y \ce x_y$
and $z \ce x_z$. Let $R = R_e(D_{IIIa})$, and note that $R(D_{IIIa}) \cong R[x]$.
As matrix factorizations over $R(D_{IIIa})$,
\begin{align*}
C_p(D_{IIIa})\{3\}
&=     \koszul{* & (x_d-x_a)(x_d-z) \\ * & (x_e-x)(x_e-y) \\ * & (z-x_b)(z-x_c)} \\
&=     \koszul{* & (x_d-x_a)(x_a-x) \\ * & (x_e-x)(x-x_f) \\ * & (x+x_d-x_a-x_b)(x+x_d-x_a-x_c)} \\
&\cong \koszul{* & (x_d-x_a)(x_a-x) \\ * & (x_e-x)(x-x_f) \\ * & x_ax_b+x_bx_c+x_cx_a - x_cx_d-x_dx_e-x_ex_f}
\text{,}
\end{align*}
where the last line is obtained from the previous one by adding the top right
and the center right entry to the bottom right entry.  Let
$p = x_e+x_f$, $q = x_ex_f$, $\alpha = x_d-x_a$ and
$\beta = x_ax_b+x_bx_c+x_cx_a - x_cx_d-x_dx_e-x_ex_f$,
so that the last line reads 
\[C_p(D_{IIIa})\{3\} \cong \koszul{* & ax_a-ax \\ * & -x^2 + px - q \\ * & b}\]
Using the second row to exclude the variable $x$, we obtain an order-two Koszul
matrix factorization over the ring $R' = R(D_{IIIa})\slash(x^2=px-q)$, which is given
explicitly (with respect to the standard decomposition of $R'$ as a free
$R$-module of rank two) as
$C_p(D_{IIIa}) \cong \mf{R_1}{R_0}{V}{U} \text{,}$
where
\begin{align*}
R_1 &= R\{3-n\} \oplus R\{5-n\} \oplus R\{3-n\} \oplus R\{5-n\} \text{,} \\
R_0 &= R\{6-2n\} \oplus R\{8-2n\} \oplus R \oplus R\{2\} \text{,}
\end{align*}
\[
A = \begin{pmatrix}
  \beta & 0 & * & * \\
  0 & \beta & * & * \\
  \alpha x_a & \alpha q & * & * \\
  -\alpha & \alpha(x_a-p) & * & *
\end{pmatrix}
\hspace{0.5cm}\text{and}\hspace{0.5cm}
B = \begin{pmatrix}
  * & * & * & * \\
  * & * & * & * \\
  \alpha x_a & \alpha q & -\beta & 0 \\
  -\alpha  & \alpha (x_a-p) & 0 & -\beta
\end{pmatrix}
\text{.}
\]

We apply the following change of basis
\[\raisebox{0.45\depth}{\xymatrix{
R_1  \ar@<0.4ex>[rr]^{V } \ar@<0.4ex>[d]^{X    } && R_0  \ar@<0.4ex>[ll]^{U } \ar@<0.4ex>[d]^{X    } \\
R'_1 \ar@<0.4ex>[rr]^{V'} \ar@<0.4ex>[u]^{X\inv} && R'_0 \ar@<0.4ex>[ll]^{U'} \ar@<0.4ex>[u]^{X\inv}}}
\text{, where }
X = \begin{pmatrix}
  1 & p-x_a & 0 & 0 \\
  0 & 0 & 0 & 1 \\
  0 & 1 & 0 & 0 \\
  0 & 0 & 1 & x_a \\
\end{pmatrix}
\text{,}
\]
\begin{align*}
R'_1 &= R\{3-n\} \oplus R\{5-n\} \oplus R\{5-n\} \oplus R\{3-n\} \text{ and} \\
R'_0 &= R\{6-2n\} \oplus R\{2\} \oplus R\{8-2n\} \oplus R \text{.}
\end{align*}
$C$ is of $q$-degree $0$; a straightforward computation shows that
\begin{align*}
U' &= \begin{pmatrix}
  \beta & * & 0 & * \\
  -\alpha  & * & 0 & * \\
  0 & * & \beta & * \\
  0 & * & \alpha(x_a-x_e)(x_a-x_f) & *
\end{pmatrix}
\hspace{0.5cm}\text{and}\\
V' &= \begin{pmatrix}
  * & * & * & * \\
  -\alpha & -\beta & 0 & 0 \\
  * & * & * & * \\
  0 & 0 & \alpha(x_a-x_e)(x_a-x_f) & -\beta
\end{pmatrix}
\text{.}
\end{align*}

We compute $\gcd(\alpha,\beta) = \gcd(\alpha, (x_d-x_c)\alpha-\beta) = \gcd((x_d-x_a), (x_b-x_e)(x_b-x_f)) = 1$, hence
by symmetry $\gcd(x_a-x_e,\beta) = 1$ and $\gcd(x_a-x_f,\beta) = 1$ as well.  Therefore,
$\gcd(\alpha,\beta) = \gcd(\alpha(x_a-x_e)(x_a-x_f), \beta) = 1$, so we may apply \autoref{stars4} to get
\begin{align*}
C_p(D_{IIIa})\{3\}
& \cong \koszul{* & -\alpha \\ * & \beta}\!\{2\} \oplus \koszul{* & \alpha(x_a-x_e)(x_a-x_f) \\ * & \beta} \\
& \cong \koszul{* & -\alpha \\ * & \beta+(x_e+x_f-x_a)\alpha}\!\{2\} \oplus \koszul{* & \alpha(x_a-x_e)(x_a-x_f) - x_a^2\beta \\ * & \beta} \\
& \cong \koszul{* & x_a-x_d \\ * & x_bx_c-x_ex_f}\!\{2\} \oplus \koszul{* & x_ax_bx_c-x_dx_ex_f \\ * & \beta}
\end{align*}

It is easy to see that the first summand is isomorphic to
$C_p(D'_{IIIb})\mfshift\{3\}$.  Denote the second
summand by $\Upsilon\{3\}$, so that we have $C_p(D_{IIIa}) \cong C_p(D'_{IIIb})\mfshift \oplus \Upsilon$.
By \autoref{reflection}, reflection along the middle strand
is given by the ring homomorphism $\bar \cdot: R \to R$,
$\bar x_a = -x_c$,
$\bar x_b = -x_b$,
$\bar x_c = -x_a$,
$\bar x_d = -x_f$,
$\bar x_e = -x_e$ and
$\bar x_f = -x_d$.
Since $\bar \Upsilon \cong \Upsilon$ by \autoref{stars2} under this isomorphism, we obtain that
$C_p(D'_{IIIa}) \cong C_p(D_{IIIb})\mfshift \oplus \Upsilon$, which implies claim (e).
\bigskip

(f) This follows immediately from (a) and (d).
\end{proof}

We will collectively refer to diagrams of type $D_r$ and $D_s$ as \emph{resolved crossings}.

\begin{proof} \textit{(of \autoref{\hmf2})}
We will prove the theorem by repeatedly reducing $C_p(D)$ according to
\autoref{reduce}.
At each step, we get a complex of matrix factorizations whose
underlying graded object is $\bigoplus_i C_p(D_i)$ for some
collection of singular diagrams in braid form. Following Wu~\cite{wu}, we
define a complexity function on singular braids by
$i(D) = \sum_j i_j$ where $j$ runs over all resolved crossings
in the diagram and $i_j$ is $1$ for an oriented smoothing and one plus the
number of strands to the left of the crossing for a singular crossing.
We show that each step of the reduction process decreases either
the maximum complexity of diagrams $D_i$ or the number of diagrams that have
maximum complexity.  This reduction can be performed as long as the maximum
complexity is greater than $1$.  The only connected diagrams of complexity $1$
are $D_r$ and $D_s$, so if the maximum complexity is $1$,
$C_p(D)$ is the direct sum of shifts of $C_p(D_r)$ and $C_p(D_s)$ and the Lemma
follows.  To perform the reduction, choose a diagram of maximum complexity.  The
Lemma below guarantees that either $D_{IIIa}$ or one of the diagrams on the left
hand side of \autoref{reduce}(a)-(d) or (f) is a subdiagram of $D$.
In the latter case we can simply replace the complex on the left hand side by the
the one on the right-hand side; notice that this reduces the number of diagrams
of this complexity.  If there is a subdiagram of type $D_{IIIa}$, we are given
a complex of the form
\[
\dots
C^{k-1} \xrightarrow{\smatrix{\alpha \\ \beta}}
C^k \oplus C_p(D_{IIIa}) \xrightarrow{\smatrix{\gamma & \delta}}
C^{k+1} \dots \text{,}
\]
which is (up to a grading shift) isomorphic in $\K(\hmf_w(R))$ to
\[
\dots
C^{k-1} \oplus C_p(D_{IIIb}) \xrightarrow{\smatrix{\alpha \\ \beta \\ 1}}
C^k \oplus C_p(D_{IIIa}) \oplus C_p(D_{IIIb}) \xrightarrow{\smatrix{\gamma & \delta & 0}}
C^{k+1} \dots \text{,}
\]
which is in turn isomorphic to
\[
\dots
C^{k-1} \oplus C_p(D_{IIIb}) \xrightarrow{}
C^k \oplus C_p(D'_{IIIa}) \oplus C_p(D'_{IIIb}) \xrightarrow{}
C^{k+1} \dots \text{,}
\]
so we once again were able to reduce the number of diagrams of the given complexity.
\end{proof}

\begin{lemma}
If $D$ is a connected singular (open) braid diagram of complexity greater than $1$, then
it contains at least one of the following subdiagrams:
\begin{itemize}
\item[(i)] A resolved crossing of type $D_r$ or $D_s$ in rightmost  position
which is the only resolved crossing in this column,
\item[(ii)] a diagram $D_r$ which has the property that $D$ stays connected when $D_r$ is removed,
\item[(iii)] a diagram of type $D_{II}$, $D_{III}$ or $D_{IV}$.
\end{itemize}
\end{lemma}
\begin{proof}
We prove the lemma by induction on the braid index.  If the braid index is $2$
and $i(D) > 1$, then we either have a subdiagram of type $D_r$, which can be
removed without disconnecting the diagram, or we have at least two subdiagrams
of type $D_s$ and none of type $D_r$, so we can find $D_{II}$ as a subdiagram.
If the braid index is greater than $2$, we may assume that there are at least
two resolved crossings in rightmost position.  We may also assume that we
have no subdiagrams of type $D_r$ in rightmost position, since we could remove
them without disconnecting the diagram.  If two such singular crossings are
adjacent, we have found $D_{II}$ as a subdiagram. Otherwise choose the topmost
two such singular crossings and apply the induction hypothesis to the part of
the braid between those two singular crossings, giving us either a subdiagram
of the required type inside this part of the braid or, potentially after
performing an isotopy, a diagram of type $D_{III}$ or $D_{IV}$.
\end{proof}

\section{Mutation invariance of the inner tangle}\label{section-mutation}

The following simple lemma is at the heart of the proof.  We will use it to
show that invariance under mutation is essentially a property of the category
of matrix factorizations associated to 2-tangles.  The functors $\FF$
and $\GG$ are necessary to account for grading shifts; we suggest
that the reader think of them as identity functors and of $f$ as a natural
transformation in the center of the category.

\begin{lemma}\label{cone}
Let $\CC$ be an additive category and let $\FF$, $\GG$ and $\bar \cdot$ be
additive endofunctors of $\CC$, where $\bar \cdot: \CC \to \CC$ is required to
be the identity on objects and an involution on morphisms.  Furthermore, let
$f$ be a natural transformation from $\FF$ to $\GG$ and let
$\partial: \Hom_\CC(A,B) \to \Hom_\CC(\GG A,\FF B)$ be an operation defined on the
$\Hom$-sets of $\CC$ with the following properties
\begin{enumerate}
\item $\partial$ is $\Z$-linear, i.e.\ for $\phi, \psi \in \Hom_\CC(A,B)$,
$\partial (\phi - \psi) = \partial \phi - \partial \psi$.
\item For $\phi \in \Hom_\CC(A,B)$, 
$\GG(\phi - \bar \phi) = f_B \, \partial \phi$ and $\FF(\phi - \bar \phi) = \partial \phi \, f_A$.
\item Composable morphisms $\phi \in \Hom_\CC(A,B)$ and $\psi \in \Hom_\CC(B,C)$ satisfy a perturbed Leibniz rule:
$
\partial (\psi \, \phi)
= \partial \psi \, \GG      \phi + \FF \bar \psi \, \partial \phi 
= \partial \psi \, \GG \bar \phi + \FF      \psi \, \partial \phi
$.
\end{enumerate}

If $C$ is a chain complex over $\CC$ with differential $d$, then $f$
gives rise to a chain morphism $f_C: \FF C \to \GG C$.  Let
$\bar C$ be the chain complex obtained by applying $\bar \cdot$ to the
differential of $C$.  Then the mapping cones $\Cone(f_C)$ and $\Cone(f_{\bar C})$
are isomorphic.

\end{lemma}
\begin{proof}
We adopt the following conventions for the mapping cone. Let $\varepsilon_C: C
\to C$ be the identity in even homological heights and the negative of the
identity in odd heights.  Note that $\varepsilon_C$ commutes with morphisms of
even homological degree and anti-commutes with morphisms of odd degree.  Then
the mapping cone $\Cone(f_C)$ is given by $\FF C[-1] \oplus \GG C$ with differential
$\mmatrix{\FF d}{}{f_C\varepsilon_C}{\GG d}$. Since $d$ here has degree 1, it is easy
to check that this defines a differential.

We claim that the horizontal arrows in the diagram below define an isomorphism
between $\Cone(f_C)$ and $\Cone(f_{\bar C})$.
\[
\xymatrix@R=1.3cm@C=2.3cm{
\FF C[-1] \oplus \GG C
\ar[r]^{\mmatrix{I}{\partial d \, \varepsilon_C}{}{I}}
& \FF C[-1] \oplus \GG C \\
\FF C[-1] \oplus \GG C
\ar[r]_{\mmatrix{I}{\partial d \, \varepsilon_C}{}{I}}
\ar[u]^{\mmatrix{\FF d}{}{f_C\varepsilon_C}{\GG d}}
& \FF C[-1] \oplus \GG C
\ar[u]_{\mmatrix{\FF \bar d}{}{f_{\bar C}\varepsilon_{\bar C}}{\GG \bar d}}
}
\]
$\mmatrix{I}{\partial d \, \varepsilon}{}{I}$ is invertible with inverse
$\mmatrix{I}{-\partial d \, \varepsilon}{}{I}$,
so it remains to check that it defines a chain morphism, i.e.\ that
$\mmatrix{\FF \bar d}{}{f\varepsilon}{\GG \bar d} \mmatrix{I}{\partial d \, \varepsilon}{}{I}
=\mmatrix{I}{\partial d \, \varepsilon}{}{I} \mmatrix{\FF d}{}{f\varepsilon}{\GG d}$.
We have $\deg(\FF d) = \deg(\GG d) = 1$, $\deg(\partial d) = 1 + (-1) - 0 = 0$ and
$\deg f = 0 + 0 - (-1) = 1$, so this follows from
$\FF \bar d = \FF d - \partial d \, f = \FF d + \partial d \, \varepsilon f \varepsilon$, from
$\FF \bar d \, \partial d \, \varepsilon = -\partial d \, \GG d \, \varepsilon = \partial d \, \varepsilon \, \GG d$ and from
$f \varepsilon \partial d \, \varepsilon + \GG \bar d = f \, \partial d + \GG \bar d = \GG d$, where the second
identity follows from $0 = \partial (d^2) = \partial d \, \GG d + \FF \bar d \, \partial d$.
\end{proof}

\begin{lemma}
Let $R = \Q[x_a,x_b,x_c,x_d]\slash(x_a+x_b=x_c+x_d)$ and let $\bar \cdot$ be
the ring homomorphism defined by
$\bar x_a = -x_b$, $\bar x_b = -x_a$, $\bar x_c = -x_d$ and $\bar x_d = -x_c$,
which induces an involution functor on $\hmf_2$.  Let $\FF$ be the
grading shift functor $\{2\}$ and $\GG$ be the identity functor.
Then there is a differential $\partial$ on the morphism spaces of $\hmf_2$
satisfying the hypothesis of the previous lemma.
\end{lemma}

\begin{proof}
$\bar \cdot: R \to R$ is well-defined since $\overline {x_a+x_b} = -x_b-x_a =
-x_d-x_c = \overline{x_c+x_d}$.  R is isomorphic to the polynomial ring in
$x_a$, $x_b$ and $x_c$.  Substituting $x_b = -x_a$ in any expression of the
form $r - \bar r$, we obtain 0, hence $r - \bar r$ is divisible by $x_a+x_b$
and we may define $\partial$ on $R$ by $\partial r = \frac {r - \bar r}
{x_a+x_b}$.  Viewing the ring $R$ as an additive category with one element, it
is straightforward to check that $\partial$ satisfies the hypothesis of
\autoref{cone}.

The differential descends to a differential on $\hmf_2$.  First note that objects in $\hmf_2$ are direct
sums of one-term Koszul factorizations $K(u;v)$ with potential $w = p(x_c)+p(x_d)-p(x_a)-p(x_b)$.  It 
follows from the proof of the one-crossing case of \autoref{reflection} that
$\partial w = 0$ and $\partial v = 0$ for the two choices of $v$, that
is $v = x_c-x_a$ and $v = (x_c-x_a)(x_c-x_b)$.  This implies 
$0 = \partial w = \partial u \, v + \bar u \, \partial v = \partial u \, v$,
hence $\partial v = 0$ since $R$ does not have zero divisors.  We define the
differential of a morphism of such matrix factorizations,
\[\xymatrix@C=2cm{
R\{\frac{\deg v' - \deg u'}2\} \ar@<0.4ex>[r]^-{v'} &
R \ar@<0.4ex>[l]^-{u'} \\
R\{\frac{\deg v - \deg u}2 + \deg z\} \ar@<0.4ex>[r]^-{v} \ar[u]^{y} &
R\{\deg z\} \ar@<0.4ex>[l]^-{u} \ar[u]_{z}
}\]
to simply be
\[\xymatrix@C=2cm{
R\{\frac{\deg v' - \deg u'}2 + 2\} \ar@<0.4ex>[r]^-{v'} &
R\{2\} \ar@<0.4ex>[l]^-{u'} \\
R\{\frac{\deg v - \deg u}2 + \deg z\} \ar@<0.4ex>[r]^-{v} \ar[u]^{\partial y} &
R\{\deg z\} \ar@<0.4ex>[l]^-{u} \ar[u]_{\partial z}
}\]
This is a morphism of matrix factorizations since
$\partial y \, u = \partial(yu) = \partial(u'z) = u' \, \partial z$ and
$\partial z \, v = \partial(zv) = \partial(v'y) = v' \, \partial y$.

Since any null-homotopic morphism
\[\xymatrix@C=2cm{
R\{\frac{\deg v' - \deg u'}2\} \ar@<0.4ex>[r]^-{v'} &
R \ar@<0.4ex>[l]^-{u'} \\
R\{\frac{\deg w}2 + \deg h\} \ar@<0.4ex>[r]^-{v} \ar[u]^{u' h + k v} &
R\{\deg h + \deg u\} \ar@<0.4ex>[l]^-{u} \ar[u]_{h u + v' k}
}\]
is sent to the null-homotopic morphism
\[\xymatrix@C=2cm{
R\{\frac{\deg v' - \deg u'}2 + 2\} \ar@<0.4ex>[r]^-{v'} &
R\{2\} \ar@<0.4ex>[l]^-{u'} \\
R\{\frac{\deg w}2 + \deg h\} \ar@<0.4ex>[r]^-{v} \ar[u]^{u' \, \partial h + \partial k \, v} &
R\{\deg h + \deg u\} \ar@<0.4ex>[l]^-{u} \ar[u]_{\partial h \, u + v' \, \partial k}
}\]
$\partial$ descends to a differential on $\hmf_2$.

The natural transformation $\phi$ is given by
\[\xymatrix@C=2cm{
R\{\frac{\deg v - \deg u}2\} \ar@<0.4ex>[r]^-{v} &
R \ar@<0.4ex>[l]^-{u} \\
R\{\frac{\deg v - \deg u}2 + 2 \} \ar@<0.4ex>[r]^-{v} \ar[u]^{x_a+x_b} &
R\{2\} \ar@<0.4ex>[l]^-{u} \ar[u]_{x_a+x_b}
}\]
Since we can view (representatives of) morphisms in $\hmf_2$ as pairs of elements of $R$,
the fact that $R$ satisfies the hypothesis of \autoref{cone} implies that $\hmf_2$ does
as well.
\end{proof}

\section{Proof of the main Theorem}\label{section-proof}
Before we can finish the proof, we need to borrow another Lemma from \cite{rasmussen}.
\begin{lemma}(Lemma 5.16 in \cite{rasmussen})\label{homotopy}
Let $D$ be the diagram of a single crossing with endpoints as in \autoref{figure-orientations}(a).
Then the maps $x_b: C_p(D)\{2\} \to C_p(D)$ and $x_c: C_p(D)\{2\} \to C_p(D)$ are homotopic.
Since $x_a + x_b = x_c + x_d$, this of course implies that $x_a$ and $x_d$ are homotopic as well.
\end{lemma}
\begin{proof}
Let $d_+: C_p(D_r) \to C_p(D_s)$ be the differential of a positive crossing and
$d_-: C_p(D_s) \to C_p(D_r)$ be the differential of a negative crossing. Clearly,
$d_-d_+ = x_c-x_b: C_p(D_r) \to C_p(D_r)$ and $d_+d_- = x_c-x_b: C_p(D_s) \to C_p(D_s)$, so 
$d_\mp$ is a null-homotopy for $x_c-x_b: C_p(D_\pm) \to C_p(D_\pm)$.  We ignored
$q$-gradings above, the reader can easily check that the proof applies in the
graded setting as well.
\end{proof}

We are now ready to prove \autoref{main}.  Given a pair of mutants $L_1$ and $L_2$, 
we may assume, by \autoref{braid-form}, that the mutation is realized 
as a mutation of type $R_z$ whose inner tangle diagram $D$ is in braid form.
By \autoref{\hmf2}, there is an object $C$ in $\K(\hmf_2)$ such that
$C_p(D) \cong C$ in $\K(\hmf_w(R))$.  Applying the ring isomorphism $\bar \cdot$,
we obtain an isomorphism $\overline{C_p(D)} \cong \bar C$, hence by
\autoref{reflection}, $C_p(\bar D) \cong \bar C$.  Applying \autoref{cone}, we
obtain that $\Cone(x_a+x_b: C\{2\} \to C)$ is isomorphic in $\hmf_2$,
and hence in $\hmf_w(R)$ to $\Cone(x_a+x_b: \bar C\{2\} \to \bar C)$.  Taking
the tensor product over $\Q[x_b, x_c, x_d]$ with the complex associated to the
outer tangle, we get an isomorphism 
\[ \Cone(x_a+x_b: C_p(L^\circ_1)\{2\} \to C_p(L^\circ_1)) \cong 
   \Cone(x_a+x_b: C_p(L^\circ_2)\{2\} \to C_p(L^\circ_2)) \]
by \autoref{gluing}, where $L^\circ_1$ and $L^\circ_2$ denote $L_1$ and $L_2$ cut open at $a$,
respectively.  Because we consider only positive mutation, $x_a$
and $x_b$ lie on the same component of both $L_1$ and $L_2$, so $x_a$ and $x_b$
are homotopic by repeated application of \autoref{homotopy}. Hence we get an
isomorphism
\[
\Cone(2x_a: C_p(L^\circ_1)\{2\} \to C_p(L^\circ_1)) \cong
\Cone(2x_a: C_p(L^\circ_2)\{2\} \to C_p(L^\circ_2))
\]
and thus
\[
\Cone(x_a: C_p(L^\circ_1)\{2\} \to C_p(L^\circ_1)) \cong
\Cone(x_a: C_p(L^\circ_2)\{2\} \to C_p(L^\circ_2))
\text{.}\]
Tensoring with $K(p'(x_a); 0)$ we get
$
\Cone(\widehat C_p(L^\circ_1)) \cong
\Cone(\widehat C_p(L^\circ_2))
$, which implies that the reduced homologies of $L_1$ and $L_2$ are isomorphic.
\qed

\begin{proof}[Proof (of \autoref{homfly-pt})]
This follows directly from Theorem 1 in \cite{rasmussen}, which asserts that
for sufficiently large $n$, the $\sln$ homology of a knot is a regraded version
of its HOMFLY-PT homology.  It is clear that we can recover the triple
grading of HOMFLY-PT homology by choosing $n$ large enough.
\end{proof}

\end{document}